%% file: main.tex
\documentclass{article}
\usepackage[T1]{fontenc}
\usepackage{iclr2027_conference,times}
\usepackage{amsmath,amssymb,amsthm,booktabs,graphicx,url,float}
\usepackage{hyperref}

\hypersetup{hidelinks,
  pdftitle={Conformal Coverage for Time Series: Validity and Inference},
  pdfauthor={Percy S. Zhai, Maggie Cheng, Wei Biao Wu}}
\newtheorem{theorem}{Theorem}
\newtheorem{proposition}[theorem]{Proposition}
\newtheorem{lemma}[theorem]{Lemma}

\newtheorem{assumption}{Assumption}
\theoremstyle{remark}

\newcommand{\E}{\mathbb E}
\renewcommand{\Pr}{\mathbb P}
\newcommand{\Var}{\operatorname{Var}}
\newcommand{\Cov}{\operatorname{Cov}}
\newcommand{\ind}{\mathbf 1}
\newcommand{\qhat}{\widehat q_n}
\newcommand{\coverage}{\widehat{\mathrm{Cov}}_{n,m}}

\title{Conformal Coverage for Time Series:\\Validity and Inference}
\author{Percy S. Zhai\textsuperscript{1}\qquad
Maggie Cheng\textsuperscript{2}\qquad
Wei Biao Wu\textsuperscript{1}\\
\normalfont\textsuperscript{1}University of Chicago (uchicago.edu)\\
\normalfont\textsuperscript{2}Illinois Institute of Technology (iit.edu)\\
{\normalfont\small\texttt{percy.zhai@chicagobooth.edu}\quad
\texttt{maggie.cheng@iit.edu}\quad\texttt{wbwu@uchicago.edu}}}
\iclrfinalcopy

\begin{document}
\raggedbottom
\setlength{\parskip}{5pt}
\maketitle
\fancyhead{}
\renewcommand{\headrulewidth}{0pt}

\begin{abstract}
Conformal prediction provides marginal coverage guarantees, yet
practitioners may wonder if the observed coverage is truly abnormal or
consistent with sampling variation. Inference for realized coverage has
received comparatively little attention, especially for time series.
We study split conformal prediction with adjacent calibration and test
sets of temporally dependent data. Using the functional dependence
measure, we derive non-asymptotic bounds on marginal coverage error
without mixing assumptions, which can be difficult to verify and may
fail even for simple short-memory models. We establish a Bahadur
representation to derive, to our knowledge, the first central limit
theorem for realized coverage of split conformal prediction under
temporal dependence. A consistent block-based estimator of the standard
error yields an asymptotically justified test. We further study long-memory time
series, which remain understudied in conformal prediction. For Gaussian
linear processes, we show how very strong temporal dependence can lead
to a non-Gaussian limiting law of realized coverage and establish
block-sampling inference with an estimated normalization. The resulting
theory explains how temporal dependence changes coverage uncertainty.
\end{abstract}

\input{introduction}
\input{setup}
\input{general}
\input{long_memory}

\input{experiments}

\section{Discussion}
\label{sec:discussion}
We develop non-asymptotic marginal coverage bounds and asymptotic inference
for realized coverage under functional dependence.
The marginal bounds broaden applicability beyond mixing rather than improve
the $O(\log n/n)$ coverage-loss rate available under geometric
$\beta$-mixing \citep{barber2026predictive}. For Gaussian linear
processes, block sampling extends this inference to long memory, including
non-Gaussian limiting distributions.
These results provide a reference for assessing whether an observed departure
from nominal coverage is consistent with sampling variation. The simulations
and electricity-demand example show that accounting for temporal dependence
can substantially change coverage uncertainty and the conclusions of coverage
tests. Our analysis assumes stationary data and a predictor trained independently
of the adjacent calibration and test windows. Extending realized-coverage
inference to nonstationary data and predictors trained or updated on the
same series is beyond the scope of this paper and is left for future work.

\clearpage
\setlength{\parskip}{0.5pc}
\bibliography{refs}
\bibliographystyle{iclr2027_conference}

\clearpage
\appendix
\input{app_general}
\input{app_long_memory}
\input{app_blocks}
\input{app_non_gaussian}
\input{experiments_appendix}
\input{app_real_data}
\end{document}

%% file: introduction.tex
\section{Introduction}
\label{sec:introduction}

Conformal prediction constructs prediction sets with a prescribed
coverage level without specifying a parametric model for the response
distribution \citep{vovk2005algorithmic,lei2018distribution}. For time
series, the observations used to calibrate the prediction sets and the
future observations are temporally dependent. This raises two questions.
First, does the prediction set still attain the nominal marginal
coverage? Second, how should we assess the fraction of observations
covered over a future test window? The purpose of this paper is to
answer these questions for split conformal prediction, with particular
attention to inference for realized coverage under long memory.

Consider prediction intervals for electricity demand with nominal
coverage of $90\%$, but realized coverage of $85\%$ over the next
evaluation window. Is this shortfall evidence of miscalibration, or can
it be explained by sampling variation under persistent forecast errors?
Answering this question requires the distribution of realized coverage,
which under long memory may fluctuate on a slower scale and need not
have a Gaussian limit.

Existing results establish conformal validity beyond exchangeability
and provide quantitative guarantees under temporal dependence
\citep{barber2023beyond,chernozhukov2018exact,oliveira2024split,barber2026predictive}.
Many time-series guarantees use mixing conditions \citep{xu2021conformal}. These conditions can
be difficult to verify and may fail even for elementary short-memory
models \citep{andrews1984nonstrong,han2023probability}. General validity frameworks need not
assume mixing, but their bounds still require suitable control of
dependence. We use the functional dependence measure (FDM) introduced
by \citet{wu2005nonlinear}, which describes the effect of changing an
innovation on future observations. This gives non-asymptotic bounds
on marginal coverage error for processes outside mixing assumptions.

Marginal validity alone does not determine the distribution of realized
coverage. \citet{oliveira2024split} also provide concentration bounds
for the covered fraction over a test window. Such bounds control large
deviations, whereas inference based on a limiting distribution provides
a reference law for testing an observed departure from the nominal
level. \citet{gazin2024asymptotics} studies this distributional question
under iid and exchangeable assumptions. For temporally dependent scores,
we establish a Bahadur representation of the calibration cutoff and
then derive a central limit theorem (CLT) for realized coverage.
The limiting variance is generally unknown; a block-based estimator
makes the corresponding standard error available in practice.

Related coverage-inference targets should be distinguished.
\citet{ramos2026transported} study the coverage probability conditional
on calibration, and \citet{cotton2026conformal} studies variation in
the marginal probability content of the estimated cutoff with an
independent test draw. Neither is the realized fraction over an adjacent
dependent test window. \citet{retzlaff2025backtesting} develop conformal
backtests for marginal and conditional calibration, including violation
counts, but do not derive the split-conformal realized-coverage CLT
studied here. \citet{halkiewicz2026conformal} uses physical dependence
and Bahadur expansions for rolling-origin conformal prediction; our
claim is therefore not the first use of FDM in conformal inference.
Online procedures that update their prediction rule to control
long-run coverage address a different setting
\citep{gibbs2024conformal}. We calibrate once and evaluate the resulting
prediction sets over the next window.

Our general FDM conditions allow long memory in the observations in
principle, but do not describe all degrees of persistence in the
coverage indicators. We therefore study Gaussian linear processes
separately. Long-memory time series remain understudied in conformal
prediction. Earlier theory does permit some long-memory settings;
for example, \citet{ren2024conformal} allows long-memory stochastic
volatility in a high-frequency event-study framework. This does not
provide a distributional theory for realized coverage over growing,
adjacent calibration and test windows. In our setting, absolute-error
scores yield Gaussian or non-Gaussian coverage limits, depending on
the strength of memory. Classical empirical-process and quantile
theory supplies the tools for identifying these limits
\citep{dehling1989empirical,coeurjolly2008bahadur}.

The non-Gaussian regime requires more than estimating a standard error.
We use the block-sampling method of \citet{zhang2013block} to estimate
the coverage distribution. Each historical block repeats calibration
followed by evaluation on an adjacent subwindow, with the same length
ratio as the target problem. The resulting block distribution permits
inference without a Gaussian approximation. Our proof combines the
coverage expansion with innovation coupling, connecting this method
to the functional-dependence approach used in the general results.

Our contributions can be summarized as follows:
\begin{enumerate}
\item \textbf{Marginal coverage.} We derive non-asymptotic marginal
coverage bounds using FDM, broadening applicability beyond mixing-based
assumptions, and allowing long-memory processes in principle.
\item \textbf{Inference for realized coverage.} We establish a Bahadur
representation and, to our knowledge, the first central limit theorem
for realized coverage of split conformal prediction under temporal
dependence. We provide a consistent block-based standard-error estimator
for asymptotically valid coverage testing.
\item \textbf{Long-memory inference.} We provide the first distributional
theory and block-sampling inference for realized coverage of split
conformal prediction under long memory that we are aware of.
For long-memory Gaussian linear processes, we derive central and
non-central limits. Building on \citet{zhang2013block}, we estimate the
memory parameter from calibration residuals and justify the resulting
normalization, yielding asymptotically valid tests across all three
long-memory regimes.
\end{enumerate}

Section~\ref{sec:setup} introduces the setup and functional dependence.
Section~\ref{sec:general} presents the general marginal-coverage bound,
Bahadur representation, CLT, and standard-error estimator.
Section~\ref{sec:long-memory} studies the long-memory linear model and
the block-sampling method. Proofs and supporting technical results
are in the appendix.

%% file: setup.tex
\section{Setup}
\label{sec:setup}

Let $Z_t=(X_t,Y_t)$ be a stationary time series, where
$X_t\in\mathbb R^d$ and $Y_t\in\mathbb R$. A predictor
$\widehat\mu$ trained on an independent dataset is treated as a fixed
function on the calibration and test sets. We use the nonconformity
score $V_t=s(X_t,Y_t)=|Y_t-\widehat\mu(X_t)|$; the process assumptions
below apply with this predictor fixed.

The calibration set consists of $V_1,\ldots,V_n$. For a nominal
coverage level $1-\alpha$, where $0<\alpha<1$, define
\begin{equation}
 k_n=\lceil(n+1)(1-\alpha)\rceil,\qquad
 \qhat=V_{(k_n)},\qquad
 \widehat C_n(x)=[\widehat\mu(x)-\qhat,\widehat\mu(x)+\qhat],
 \label{eq:procedure}
\end{equation}
where $V_{(k)}$ is the $k$th order statistic of the calibration set.
We use $\qhat=\infty$ if $k_n>n$. The estimated cutoff is
unchanged over the adjacent test set, indexed by $n+1,\ldots,n+m$.

This paper concerns two quantities associated with this prediction
interval. First, for a new data point, is the marginal coverage rate
$\Pr\{Y_{n+1}\in\widehat C_n(X_{n+1})\}$ close to $1-\alpha$?
Second, what is the asymptotic distribution of realized coverage on
the adjacent test set? Its \emph{realized coverage} is the fraction of
test responses covered:
\begin{equation}
 \coverage=\frac1m\sum_{t=n+1}^{n+m}
       \ind\{Y_t\in\widehat C_n(X_t)\}
 =\frac1m\sum_{t=n+1}^{n+m}\ind\{V_t\le\qhat\}.
 \label{eq:coverage-target}
\end{equation}

We next describe the functional dependence measure (FDM), the key
theoretical device of this paper \citep{wu2005nonlinear}. Write
$Z_t=G(\ldots,\varepsilon_{t-1},\varepsilon_t)$ and
$V_t=H(\ldots,\varepsilon_{t-1},\varepsilon_t)$, where $G$ and $H$
are measurable functions and $(\varepsilon_t)_{t\in\mathbb Z}$ are
iid random innovations, possibly vector-valued. These representations
are assumed well defined; the same function is used at every time.
Replace $\varepsilon_0$ by an independent copy $\varepsilon_0'$,
leaving all other innovations unchanged, and denote the resulting
variables by $Z_k^*$ and $V_k^*$. For $r\ge1$, define
\begin{equation}
 \delta_r^Z(k)=\|Z_k-Z_k^*\|_r,\qquad
 \delta_r^V(k)=\|V_k-V_k^*\|_r,\qquad k\ge0,
 \label{eq:fdm}
\end{equation}
where $\|U\|_r=(\E|U|^r)^{1/r}$, using the Euclidean norm for
vectors. We use $r=2$ unless stated otherwise. FDM measures how much
changing one innovation affects an observation $k$ steps later.
Faster decay means that this effect disappears more quickly; geometric
decay is typical of models satisfying geometric moment contraction.
The coupling follows the data-generating mechanism and can often be
bounded directly from a model's recursion or linear coefficients.

For a scalar process with finite variance, we call its memory short
when its autocovariances are absolutely summable. The long-memory
processes considered here have eventually positive, nonsummable
autocovariances. Summability of the $L^2$ FDM is a sufficient condition
for short memory, although failure of that condition alone does not
establish long memory. Long-memory models are used in financial and
geophysical applications \citep{zhang2013block}. Such dependence can change both the
normalization and the limiting distribution of sample statistics,
making the usual root-$n$ Gaussian approximation inappropriate
\citep{wu2003empirical}.

In the conformal prediction setup, we are concerned with the FDM of
score-threshold indicators:
\begin{equation}
 \theta_k=\sup_{v\in\mathbb R}
  \|\ind\{V_k\le v\}-\ind\{V_k^*\le v\}\|_2,
 \qquad \Theta=\sum_{k=0}^{\infty}\theta_k.
 \label{eq:indicator-fdm}
\end{equation}

Mixing conditions instead compare events in entire past and future
sigma-fields. Verifying their decay can be difficult, and even
elementary autoregressions with discrete innovations can fail to be
strongly mixing \citep{andrews1984nonstrong,han2023probability,wu2005nonlinear}.

The temporal dependence of the observations, nonconformity scores,
and threshold indicators need not be the same. Their FDMs are related
in a model-specific way; Lemma~\ref{lem:general-transfer} in
Appendix~\ref{app:general-tools} gives useful sufficient bounds.
Section~\ref{sec:general} assumes summable indicator FDM, which ensures
short memory of the threshold indicators while allowing long-memory
observations.

%% file: general.tex
\section{General time series}
\label{sec:general}

We study marginal coverage and realized-coverage inference under functional
dependence. Let $F(v)=\Pr(V_t\le v)$ be the marginal CDF of the
nonconformity scores defined in Section~\ref{sec:setup}.

\begin{assumption}[Nonconformity scores and functional dependence]
\label{ass:general}
The marginal CDF $F$ is continuous, and the threshold indicators have
summable functional dependence: $\Theta<\infty$, with $\Theta$ defined
in \eqref{eq:indicator-fdm}.
\end{assumption}

This summability condition implies short memory of each
threshold-indicator process and can be checked directly even when mixing
fails. Consider the stationary recursion
$Y_t=(Y_{t-1}+\varepsilon_t)/2$, with iid
$\varepsilon_t\sim\operatorname{Bernoulli}(1/2)$, and the fixed predictor
$\widehat\mu\equiv1/2$. Then $Y_t$ is uniform on $[0,1]$, the
nonconformity scores $V_t=|Y_t-1/2|$ are uniform on $[0,1/2]$, and
$\theta_k\le\min\{1,2^{(1-k)/2}\}$. Although the recursion has
geometrically decaying autocovariances,
$V_{t-1}=|2V_t-1/2|$: an arbitrarily distant future score determines the
present score, so the scores are not even $\alpha$-mixing. This is the
type of distinction between mixing and functional dependence discussed
by \citet{wu2005nonlinear,han2023probability}.
Appendix~\ref{app:general-examples}
verifies these claims and the regularity needed for the results below.

Assumption~\ref{ass:general} also allows long-memory observations in
principle. For example, let $X_t=M_t$ and $Y_t=M_t+\eta_t$, where
$M_t$ is a stationary causal long-memory process and $(\eta_t)$ is iid,
independent of $M_t$, with a continuous absolute-value distribution.
The fixed predictor $\widehat\mu(x)=x$ removes the long-memory component
and gives iid nonconformity scores $V_t=|\eta_t|$, satisfying
Assumption~\ref{ass:general}. Section~\ref{sec:long-memory} studies
Gaussian linear models beyond this assumption, including regimes where
the nonconformity scores and their threshold indicators have long memory.

Let $q=F^{-1}(1-\alpha)=\inf\{v:F(v)\ge1-\alpha\}$ be the
population $(1-\alpha)$-quantile, and write $f$ for the marginal density
where it exists. The following theorem provides two non-asymptotic
bounds on marginal coverage error.

\begin{samepage}
\begin{theorem}[Non-asymptotic marginal coverage bounds]
\label{thm:marginal}
Under Assumption~\ref{ass:general}, the following bounds hold.

\textnormal{(i)} If
$2/n+(2\Theta^2/n)^{1/3}<\min\{\alpha,1-\alpha\}$,
then the split-conformal prediction set satisfies
\begin{equation}
 \left|\Pr\{Y_{n+1}\in\widehat C_n(X_{n+1})\}-(1-\alpha)\right|
 \le \frac2n+3\left(\frac{\Theta^2}{4n}\right)^{1/3}.
 \label{eq:marginal-bound}
\end{equation}
\textnormal{(ii)} Suppose additionally that, for some $e,c_f,L>0$,
$F$ has a density $f(v)\ge c_f$ on $[q-e,q+e]$, and the conditional
distribution of $V_t$ given the innovation past
$\mathcal F_{t-1}=\sigma(\varepsilon_j:j\le t-1)$ has a density
bounded by $L$ on this interval almost surely. If $2/n\le c_fe/2$, then
\begin{equation}
 \left|\Pr\{Y_{n+1}\in\widehat C_n(X_{n+1})\}-(1-\alpha)\right|
 \le \frac{2L}{c_fn}+\frac{2\sqrt2L\Theta}{c_f\sqrt n}
          +\frac{8\Theta^2}{nc_f^2e^2}.
 \label{eq:marginal-bound-smooth}
\end{equation}
Both bounds hold at every deterministic future index $t>n$.
\end{theorem}
\nopagebreak[4]
\noindent\emph{Proof: Appendix~\ref{app:general-marginal}.}
\end{samepage}

These bounds account for calibration error while allowing temporal
dependence between the calibration set and immediately adjacent test points.
Part~\textnormal{(i)} is density-free; part~\textnormal{(ii)} gives
an $O(n^{-1/2})$ bound under an additional local conditional-density
condition. That condition is used only in part~\textnormal{(ii)}.

To study the asymptotic distribution of realized coverage, we add local
regularity of the marginal distribution at the target quantile. Write
$F_n(v)=n^{-1}\sum_{t=1}^n\ind(V_t\le v)$ for the empirical CDF of
the calibration nonconformity scores.

\begin{assumption}[Quantile regularity]
\label{ass:quantile}
There are $e>0$ and $0<c_f\le C_f<\infty$ such that $F$ has a density
with $c_f\le f(v)\le C_f$ for $|v-q|\le e$. The density is continuous
at $q$. For distinct integer times $s,t$, $\Pr(V_s=V_t)=0$.
\end{assumption}

The Bahadur representation replaces the random calibration cutoff by an
indicator average at the population quantile. This is the step needed
to transfer a fixed-threshold limit to conformal coverage.

\begin{samepage}
\begin{theorem}[Bahadur representation]
\label{thm:bahadur}
Under Assumptions~\ref{ass:general}--\ref{ass:quantile},
\begin{align}
 F(\qhat)-(1-\alpha)
 &=-\frac1n\sum_{t=1}^n
       \{\ind(V_t\le q)-(1-\alpha)\}+o_{\Pr}(n^{-1/2}),
       \label{eq:general-probability-bahadur}\\
 \qhat-q
 &=-\frac1{nf(q)}\sum_{t=1}^n
       \{\ind(V_t\le q)-(1-\alpha)\}+o_{\Pr}(n^{-1/2}).
       \label{eq:general-bahadur}
\end{align}
\end{theorem}
\nopagebreak[4]
\noindent\emph{Proof: Appendix~\ref{app:general-bahadur}.}
\end{samepage}

Bahadur representations under dependence are classical
\citep{wu2005bahadur}; here the proof uses the exact conformal rank
and a local empirical-process bound under the stated indicator FDM.
The following theorem is the main result of this section. The Bahadur
representation allows us to establish asymptotic normality of realized
coverage.

\begin{theorem}[Asymptotic normality of realized coverage]
\label{thm:realized}
Under Assumptions~\ref{ass:general}--\ref{ass:quantile}, with
$n,m\to\infty$, the series
$\sigma_{\rm cov}^2=\sum_{h\in\mathbb Z}
\Cov\{\ind(V_0\le q),\ind(V_h\le q)\}$
is absolutely convergent, and
\begin{equation}
 \sqrt{\frac{nm}{n+m}}\{\coverage-(1-\alpha)\}
 \xrightarrow{d}N(0,\sigma_{\rm cov}^2).
 \label{eq:general-coverage-clt}
\end{equation}
The zero-variance Gaussian is allowed in this statement.
\end{theorem}
\nopagebreak[4]
\noindent\emph{Proof: Appendix~\ref{app:general-clt}.}

Calibration uncertainty and test variation both contribute:
the standard-error scale is $\sigma_{\rm cov}\sqrt{1/n+1/m}$.
Positive accumulated lagged indicator covariance increases this
standard error relative to iid scores. For iid nonconformity
scores, the classical first-order calibration and test variance contributions
are $\alpha(1-\alpha)/n$ and $\alpha(1-\alpha)/m$, respectively,
giving the standard error $\sqrt{\alpha(1-\alpha)(1/n+1/m)}$.
The limiting variance $\sigma_{\rm cov}^2$ is generally unknown, so
feasible inference requires estimating it. We use nonoverlapping blocks
of calibration indicators at the estimated cutoff. With a deterministic
block length $\ell=\ell_n$ and number
of blocks $B=\lfloor n/\ell\rfloor$, define
\begin{equation}
 \overline I_j=\frac1\ell\sum_{t=(j-1)\ell+1}^{j\ell}\ind(V_t\le\qhat),
 \qquad
 \overline I=\frac1B\sum_{j=1}^B\overline I_j,
 \qquad
 \widehat\sigma_{\rm cov}^{\,2}=\frac\ell B\sum_{j=1}^B(\overline I_j-\overline I)^2.
 \label{eq:general-batch-variance}
\end{equation}
The estimator uses the $B$ complete blocks and divisor $B$.
The following proposition shows that this variance estimator is consistent.

\begin{proposition}[Consistent block-based variance estimator]
\label{prop:variance}
Under Assumptions~\ref{ass:general}--\ref{ass:quantile}, suppose that
$\sum_{k\ge1}(\sum_{j\ge k}\theta_j)^{1/2}<\infty$,
$\ell\to\infty$, and $\ell/n\to0$.
Then $\widehat\sigma_{\rm cov}^{\,2}\xrightarrow{\Pr}
\sigma_{\rm cov}^2>0$.
\end{proposition}
\nopagebreak[4]
\noindent\emph{Proof: Appendix~\ref{app:general-variance}.}

The additional tail condition holds for geometrically decaying
$\theta_k$ and for $\theta_k=O(k^{-a})$ with $a>3$.
This condition is for variance estimation only.
It also guarantees a positive limiting variance.
Such block-based variance estimation is classical \citep{flegal2010batch}.

Combining Proposition~\ref{prop:variance} with
Theorem~\ref{thm:realized} gives the feasible standard error
$\widehat{\mathrm{SE}}=\widehat\sigma_{\rm cov}\sqrt{1/n+1/m}$.
After the test set is observed, the statistic
$\{\coverage-(1-\alpha)\}/\widehat{\mathrm{SE}}$ has an asymptotic
standard normal reference under the stated nonconformity-score and
calibration model. For $0<\eta<1$, a two-sided level-$\eta$ diagnostic rejects when
its absolute value exceeds $z_{1-\eta/2}$; a lower-tail diagnostic
flags unusually low coverage, where $z_u$ is the $u$-quantile of
$N(0,1)$. Before observing the test set,
$[(1-\alpha)\pm z_{1-\eta/2}\widehat{\mathrm{SE}}]\cap[0,1]$
is an asymptotic prediction interval for its random realized coverage.

For a small number $B\ge2$ of blocks, a block-based Student-$t$
approximation \citep{jones2006fixed} replaces $z_u$ by
$\sqrt{B/(B-1)}\,t_{B-1,u}$, where $t_{\nu,u}$ is the
$u$-quantile of Student's $t$ distribution with $\nu$ degrees of
freedom. Appendix~\ref{app:general-variance} derives this reference
under fixed-$B$ asymptotics; the scale factor accounts for divisor $B$.

%% file: long_memory.tex
\section{A study of long-memory linear processes}
\label{sec:long-memory}

The preceding theory assumes summability of the FDM of the
score-indicator process. Although it allows long-memory observations,
the class it covers remains restricted. We now study realized coverage
beyond this assumption, using Gaussian linear processes to describe
temporal dependence. Their coefficients allow arbitrarily strong
memory within the stationary range.

\begin{samepage}
\begin{assumption}[Gaussian linear residuals]
\label{ass:lm}
The fixed predictor has residuals
$W_t=Y_t-\widehat\mu(X_t)=\sum_{j=0}^{\infty}a_j\varepsilon_{t-j}$,
where $\varepsilon_t\stackrel{\mathrm{iid}}{\sim}N(0,1)$,
$a_0\ne0$, $a_j\sim c j^{-\beta}$, $c>0$, and $1/2<\beta<1$.
The nonconformity scores are $V_t=|W_t|$.
\end{assumption}
\end{samepage}

Write $\sigma_W^2=\E W_0^2$ and
$C_\gamma=c^2\int_0^\infty x^{-\beta}(1+x)^{-\beta}\,dx$.
Then $\Cov(W_0,W_h)\sim C_\gamma h^{1-2\beta}$, so the residual
process has long memory throughout $1/2<\beta<1$. Taking the absolute
value changes the covariance decay: both $V_t$ and its coverage
indicator have covariance of order $h^{2-4\beta}$. They have long
memory when $\beta\le3/4$, and summable covariances when $\beta>3/4$.
As memory becomes stronger, the limiting law of realized coverage
changes from Gaussian to non-Gaussian. A related phase transition
appears in covariance inference for long-memory time series
\citep{zhai2026simultaneous}. Define the corresponding rate by
\begin{equation}
r_n=
\begin{cases}
n^{-1/2},&3/4<\beta<1,\\
\sqrt{\log n/n},&\beta=3/4,\\
n^{1-2\beta},&1/2<\beta<3/4.
\end{cases}
\label{eq:lm-rate}
\end{equation}

\begin{theorem}[Bahadur representation under long memory]
\label{thm:lm-bahadur}
Under Assumption~\ref{ass:lm}, the conformal cutoff satisfies
\begin{equation}
\qhat-q=-\frac1{nf(q)}\sum_{t=1}^{n}
\big\{\ind(V_t\le q)-(1-\alpha)\big\}+o_{\Pr}(r_n).
\label{eq:lm-bahadur}
\end{equation}
\end{theorem}
\nopagebreak[4]
\noindent\emph{Proof: Appendix~\ref{app:lm-bahadur}.}

The representation is proved for the exact conformal rank, using
classical Gaussian empirical-process and quantile arguments
\citep{dehling1989empirical,coeurjolly2008bahadur}.
The proof also provides an upper bound on marginal coverage error:
$\sup_{t>n}|\Pr(V_t\le\qhat)-(1-\alpha)|=O(r_n)$
(Lemma~\ref{lem:lm-scale});
the supremum concerns deterministic future indices.

The following theorem provides the asymptotic distribution of realized
coverage in the three memory regimes and identifies the transition
from a Gaussian law to a law determined by a Rosenblatt process.

\begin{theorem}[Limiting distributions of realized coverage]
\label{thm:lm-coverage}
Under Assumption~\ref{ass:lm} and $m/n\to\kappa\in(0,\infty)$:

\textnormal{(i)} If $3/4<\beta<1$, then
\begin{equation}
\sqrt n\{\coverage-(1-\alpha)\}
\xrightarrow{d}N\!\left(0,\sigma_{\rm cov}^2(1+\kappa^{-1})\right).
\label{eq:lm-central-law}
\end{equation}
Here $\sigma_{\rm cov}^2$ is the covariance sum defined in
Theorem~\ref{thm:realized} and is strictly positive in this model.

\textnormal{(ii)} If $\beta=3/4$, then
\begin{equation}
\sqrt{\frac n{\log n}}\{\coverage-(1-\alpha)\}
\xrightarrow{d}
N\!\left(0,
\left\{\frac{qf(q)C_\gamma}{\sigma_W^2}\right\}^{\!2}
(1+\kappa^{-1})\right).
\label{eq:lm-critical-law}
\end{equation}

\textnormal{(iii)} If $1/2<\beta<3/4$, then
\begin{equation}
\frac{\coverage-(1-\alpha)}{n^{1-2\beta}}
\xrightarrow{d}
\frac{qf(q)C_\gamma}{\sigma_W^2\sqrt{(3-4\beta)(4-4\beta)}}
\left\{R_H(1)-\frac{R_H(1+\kappa)-R_H(1)}{\kappa}\right\},
\label{eq:lm-noncentral-law}
\end{equation}
where $H=2-2\beta$ and $R_H$ is the variance-one Rosenblatt process,
oriented as the limit of centered-square sums.
\end{theorem}

Part~\textnormal{(i)} has the same Gaussian form as
Theorem~\ref{thm:realized}, although the indicator FDM is not summable
in this linear model. At $\beta=3/4$, the limit remains Gaussian but
requires the logarithmic normalization. In the ultra-long-memory
regime $1/2<\beta<3/4$, it becomes non-Gaussian and retains dependence
between adjacent Rosenblatt increments \citep{taqqu1975weak}.
Appendix~\ref{app:lm-laws} proves the three laws. As before, the
calibration set and test set are adjacent, with no independence
assumption between them.

We use block sampling to estimate the limiting distribution of realized
coverage, building on \citet{zhang2013block}. Within the calibration set,
we repeat calibration and adjacent evaluation on shorter windows with
the same calibration--test length ratio. An estimated normalization
makes the procedure applicable across all three regimes.
Appendix~\ref{app:lm-procedure} gives the construction.

\begin{samepage}
\begin{theorem}[Block sampling distribution consistency]
\label{thm:blocks}
Suppose Assumption~\ref{ass:lm} holds, $m/n\to\kappa\in(0,\infty)$,
$b=\lfloor n^\zeta\rfloor$ for a fixed $0<\zeta<1$, and
$a_j=j^{-\beta}\{c+O(j^{-\varphi})\}$ for some $\varphi>0$.
Let $\widehat G_{n,b}$ and $R_N$ be the block CDF and normalization
defined in Appendix~\ref{app:lm-procedure}.
Let $G$ be the CDF of the limit of
$\{\coverage-(1-\alpha)\}/R_n(\beta)$, a constant rescaling of
the corresponding law in Theorem~\ref{thm:lm-coverage}. Then
\[
\sup_{x\in\mathbb R}|\widehat G_{n,b}(x)-G(x)|
\xrightarrow{\Pr}0.
\]
\end{theorem}
\end{samepage}

The estimated quantiles give an asymptotically valid test for realized
coverage and a prediction interval for its future value; probabilities
are over the joint calibration and test sets.
Appendix~\ref{app:lm-block-proof} gives the proof.
Appendix~\ref{app:non-gaussian} discusses the additional arguments needed
for non-Gaussian innovations.

%% file: experiments.tex
\section{Experiments}
\label{sec:experiments}
We examine marginal coverage, the normal approximation for realized
coverage, and block sampling under long memory. Throughout, we use
nonconformity scores based on absolute
errors and adjacent calibration and test sets. The conformal cutoff uses the exact rank
$k_n=\lceil(n+1)(1-\alpha)\rceil$, with $\alpha=0.1$ unless stated
otherwise.

\paragraph{Study 1: marginal coverage.}
This study examines the marginal coverage guarantee in
Theorem~\ref{thm:marginal} using a nonlinear process that satisfies our
FDM conditions but violates mixing assumptions.
Consider $U_t=(U_{t-1}+B_t)/2$, with iid
$B_t\sim\operatorname{Bernoulli}(1/2)$ and independent
$U_0\sim\operatorname{Unif}(0,1)$.
Set $Y_t=U_t^2$ and $V_t=|Y_t-1/3|$.
The score process is not strongly mixing, but its indicator FDM decays
geometrically; the verification is given in
Appendix~\ref{app:study1-details}. Thus the conditions of
Theorem~\ref{thm:marginal}(i) hold, and its bound applies for sufficiently
large $n$. We apply standard split conformal calibration and estimate
next-step marginal coverage at nominal levels 90\% and 95\%.

\begin{table}[htbp]
\caption{For Study 1, marginal coverage (\%). Columns give $n$; rows give
the nominal coverage. At $n=2^4$, the 95\% procedure has an infinite cutoff.}
\label{tab:study1}
\centering
\small
\setlength{\tabcolsep}{3.6pt}
\input{experiment_assets/study1/study1_marginal_coverage.tex}
\end{table}

Table~\ref{tab:study1} illustrates
Theorem~\ref{thm:marginal}(i): standard split conformal calibration
can approach nominal coverage even for this nonmixing process.
The 90\% procedure gives only 88.80\% at $n=2^5$; at $n=2^{12}$, coverage is
90.19\% and 95.06\% for the two nominal levels.
The contribution is a guarantee under FDM conditions where mixing
assumptions fail, without changing the prediction procedure.

\paragraph{Study 2: realized-coverage inference for general time series.}
This study evaluates realized-coverage inference for general time series
under the FDM conditions of Theorem~\ref{thm:realized} and
Proposition~\ref{prop:variance}.
We consider ARCH(1) and stochastic volatility (SV), with absolute-error
scores. Both satisfy the CLT and variance-estimation conditions; their
specifications and verification are given in
Appendix~\ref{app:study2-details}.
We use $m=n$.
The variance is estimated from the calibration set using
nonoverlapping blocks of length $\ell=\lfloor n^{2/3}\rfloor$.
We compare the empirical distribution of
$\{\coverage-(1-\alpha)\}/\widehat{\mathrm{SE}}$ with its
standard normal limit, where
$\widehat{\mathrm{SE}}=\widehat\sigma_{\rm cov}\sqrt{1/n+1/m}$.
We also report two-sided rejection rates using standard normal critical
values. Rates near the 5\% and 10\% test levels indicate accurate test
calibration; nominal prediction coverage remains 90\%.

\begin{figure}[htbp]
\centering
\includegraphics[width=\linewidth]{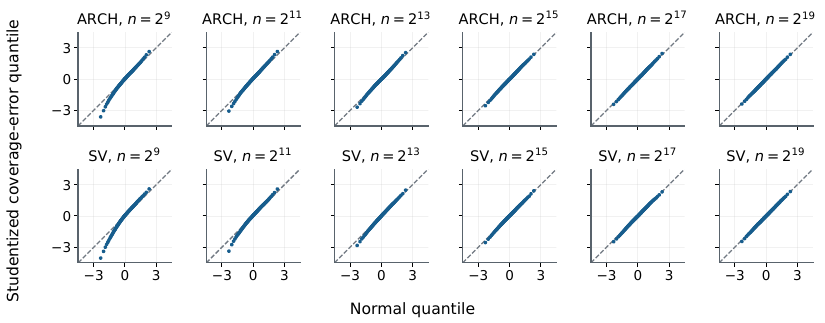}
\caption{For Study 2, normal QQ plots of studentized realized-coverage errors.
Points closer to the dashed equality line indicate closer agreement with
the standard normal distribution $N(0,1)$. No empirical recentering or
rescaling is used.}
\label{fig:study2}
\end{figure}

In Figure~\ref{fig:study2}, the points approach the dashed line as $n$
grows: the studentized realized-coverage error increasingly follows
$N(0,1)$ in both models. This supports Theorem~\ref{thm:realized} and
Proposition~\ref{prop:variance}, and the use of the estimated standard
error for coverage testing. Tail discrepancies at small $n$ lead to
over-rejection with normal critical values.
Appendix Table~\ref{tab:study2-comparisons} quantifies this: the 5\%
rejection rates fall from 10.54\% and 12.92\% at $n=2^9$ to 5.49\%
and 5.61\% at $n=2^{19}$ for ARCH and SV. The corresponding 10\%
rates at the largest $n$ are 10.77\% and 10.94\%. The normal
approximation therefore gives increasingly well-calibrated tests.

We also compare with the exchangeable standard error
$\sqrt{\alpha(1-\alpha)(1/n+1/m)}$, which substantially understates
coverage uncertainty in these models. At $n=2^{19}$, the simulated standard deviations are
1.40 and 1.97 times that benchmark for ARCH and SV; nominal 5\%
tests using it reject 16.04\% and 32.10\% of the time. Ignoring temporal
dependence would therefore lead to frequent false alarms about coverage.
The corrected $t$ reference reduces the rejection rates at $n=2^9$
to 4.75\% and 6.19\%, using the same variance estimates.
Appendix~\ref{app:study2-comparisons}, Table~\ref{tab:study2-comparisons},
gives the full comparison.

\paragraph{Study 3: block sampling method under long memory.}
This study evaluates the block sampling method for Gaussian linear
processes across the three limiting regimes in
Theorems~\ref{thm:lm-coverage} and~\ref{thm:blocks}.
Consider the Gaussian linear model
$Y_t=\sum_{j\ge0}(1+j)^{-\beta}\varepsilon_{t-j}$ with iid
$\varepsilon_t\sim N(0,1)$ and $V_t=|Y_t|$.
We set $\beta\in\{0.9,0.75,0.6\}$, covering the central Gaussian,
critical Gaussian, and non-Gaussian regimes, respectively.
The $\beta=0.6$ model has extremely strong long memory, so realized
coverage is expected to converge very slowly to $1-\alpha$: its error
scale is $n^{-0.2}$ rather than the usual $n^{-1/2}$.
For each $n\in\{2^{12},2^{14},2^{16},2^{18},2^{20}\}$, we use $m=n$
and a separate independent reference sample.
We use $h=\lfloor\sqrt n\rfloor$ for memory estimation
and pairs of adjacent blocks, each of length $b=\lfloor n^{2/3}\rfloor$,
within the calibration set. The memory regime is not supplied;
simulation details are given in Appendix~\ref{app:study3-details}.

\begin{figure}[!htb]
\centering
\includegraphics[width=\linewidth]{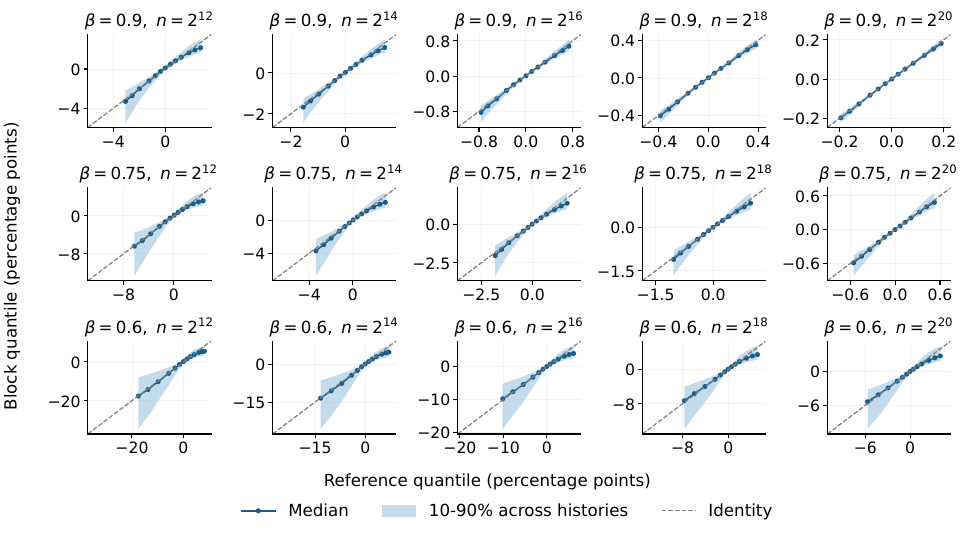}
\caption{For Study 3, block versus independent finite-$n$ reference error quantiles
(percentage points). Curves are medians across histories; shading is
the pointwise 10--90\% range, not a confidence band.
Closer agreement with the dashed equality line means more accurate
estimation of coverage-error quantiles, not agreement with a Gaussian law.}
\label{fig:study3}
\end{figure}

\begin{table}[!htb]
\caption{For Study 3, rejection rates (\%) at the 5\% test level.
BS denotes block sampling;
all methods use length $b=\lfloor n^{2/3}\rfloor$.
Rates close to 5\% indicate accurate two-sided rejection rates.}
\label{tab:study3}
\centering
\small
\setlength{\tabcolsep}{4pt}
\input{experiment_assets/study3_benchmarks_10k_20260925/benchmark_rejection_05.tex}
\end{table}

In Figure~\ref{fig:study3}, agreement with the dashed line means that
block quantiles approximate actual coverage-error quantiles from independent
reference trajectories. Together with decreasing CDF distances
(Appendix Table~\ref{tab:study3-benchmark-ks}), this supports
Theorem~\ref{thm:blocks}: block sampling estimates the distribution needed
for coverage testing, including in the non-Gaussian regime.
For $\beta=0.6$, the 5\% and 10\% rejection rates decrease from
13.85\% and 17.18\% to 5.65\% and 10.34\%, respectively, across
the sample sizes studied (Table~\ref{tab:study3} and
Appendix Table~\ref{tab:study3-benchmark10}).
For reference, Appendix~\ref{app:study3-extra} also reports $\beta=0.7$,
another ultra-long-memory case, with its QQ plots and rejection rates.

We compare block sampling with the moving-block bootstrap (MBB) and
Gaussian inference using a heteroskedasticity and autocorrelation
consistent (HAC) variance estimator on the same trajectories
over the full grid, using the same length
$b=\lfloor n^{2/3}\rfloor$ for all three methods.
Table~\ref{tab:study3} shows that the benchmarks substantially
underestimate uncertainty under stronger memory, whereas block sampling
improves test calibration. At $\beta=0.6$ and $n=2^{20}$, the 5\%
rejection rate is 5.65\% for block sampling, compared with 44.35\%
and 43.77\% for MBB and HAC.
Appendix~\ref{app:study3-benchmarks} gives the procedures, 10\%
rejection rates, and distribution comparisons.

The strong memory also makes coverage
uncertainty substantial: even at $n=2^{20}$, the reference 2.5\% and
97.5\% quantiles of realized coverage for $\beta=0.6$ are 84.29\% and
94.04\%, despite nominal prediction coverage of 90\%.

\paragraph{Real Data Analysis.}
We analyze hourly electricity-demand forecasts in New England
(Appendix~\ref{app:real-data-isne}). The realized coverage is approximately
91.2\% against a nominal level of 90\%. The exchangeability-based test
rejects nominal coverage, whereas our block sampling test does not.
This example illustrates how accounting for temporal dependence can
affect the interpretation of an observed coverage deviation.

%% file: experiment_assets/study1/study1_marginal_coverage.tex
\begin{tabular}{lrrrrrrrrr}
\toprule
Nominal & $2^{4}$ & $2^{5}$ & $2^{6}$ & $2^{7}$ & $2^{8}$ & $2^{9}$ & $2^{10}$ & $2^{11}$ & $2^{12}$ \\
\midrule
90\% & 89.61 & 88.80 & 89.48 & 89.84 & 90.30 & 90.42 & 89.94 & 90.07 & 90.19 \\
95\% & 100.00 & 94.33 & 94.03 & 94.92 & 94.91 & 94.62 & 94.89 & 94.88 & 95.06 \\
\bottomrule
\end{tabular}

%% file: experiment_assets/study3_benchmarks_10k_20260925/benchmark_rejection_05.tex
\begin{tabular}{crrrrrrrrr}
\toprule
$n=m$ & \multicolumn{3}{c}{$2^{12}$} & \multicolumn{3}{c}{$2^{16}$} & \multicolumn{3}{c}{$2^{20}$} \\
\cmidrule(lr){2-4}\cmidrule(lr){5-7}\cmidrule(lr){8-10}
$\beta$ & BS & MBB & HAC & BS & MBB & HAC & BS & MBB & HAC \\
\midrule
0.9 & 8.90 & 11.01 & 10.78 & 5.48 & 6.68 & 6.79 & 5.24 & 5.58 & 5.71 \\
0.75 & 10.09 & 17.06 & 16.07 & 5.94 & 14.59 & 14.55 & 4.27 & 13.00 & 12.87 \\
0.6 & 13.85 & 32.11 & 27.54 & 7.89 & 37.23 & 35.56 & 5.65 & 44.35 & 43.77 \\
\bottomrule
\end{tabular}

%% file: app_general.tex
\section{Proofs for general time series}
\label{app:general}

This appendix proves the general FDM results and explains how their
assumptions can be checked. Appendix~\ref{app:general-tools} develops
the empirical-CDF bounds and dependence transfer;
Appendix~\ref{app:general-examples} verifies the examples.
Appendices~\ref{app:general-marginal}, \ref{app:general-bahadur},
and~\ref{app:general-clt} then prove marginal validity, the Bahadur
representation, and the realized-coverage CLT, respectively.
Appendix~\ref{app:general-variance} proves variance-estimator consistency
and justifies both the normal and corrected $t$ references for inference.

Throughout this appendix let
$\mathcal F_j=\sigma(\varepsilon_i:i\le j)$ and
$P_jU=\E(U\mid\mathcal F_j)-\E(U\mid\mathcal F_{j-1})$.
For a random variable constructed from the innovations, a superscript
$*$ denotes replacement of the indicated innovation coordinates by an
independent copy. All nonconformity-score maps are fixed as in the main text.

\subsection{Coupling, empirical control, and dependence transfer}
\label{app:general-tools}

For a causal $U\in L^r$, $r\ge1$, with single-coordinate replacement
$U^{(j)}$, conditional integration gives
\begin{equation}
 P_jU=\E(U-U^{(j)}\mid\mathcal F_j),\qquad
 \|P_jU\|_r\le\|U-U^{(j)}\|_r.
 \label{eq:general-projection-coupling}
\end{equation}
The identity applies when $U$ is measurable with respect to
$\mathcal F_t$, $t\ge j$; innovations after $j$ are integrated out.
This is the standard projection argument associated with functional
dependence \citep{wu2005nonlinear}. The iid remote-past tail is trivial,
so every centered causal $L^2$ variable has its martingale projection
decomposition.

Let $F_N^{(a)}$ be the empirical CDF of
$V_{a+1},\ldots,V_{a+N}$, so $F_n^{(0)}=F_n$ for the calibration set.
For a deterministic threshold $v$ and fixed
lag $k$, the variables
$P_{t-k}\{\ind(V_t\le v)-F(v)\}$ are orthogonal across $t$.
Consequently Minkowski's inequality and
\eqref{eq:general-projection-coupling} give
\begin{equation}
 \left\|\sum_{t=a+1}^{a+N}\{\ind(V_t\le v)-F(v)\}\right\|_2
 \le\sum_{k\ge0}\sqrt N\,\theta_k
 =\sqrt N\,\Theta.
 \label{eq:general-empirical-l2}
\end{equation}
The finite sum in $t$ and infinite lag sum can be interchanged because
the displayed lag norms are summable. The bound is uniform in a
deterministic threshold, not an $L^2$ bound on the supremum of the
empirical process.

\begin{lemma}[Dependence transfer for fixed absolute nonconformity scores]
\label{lem:general-transfer}
Let $Z_t=(X_t,Y_t)$ be a causal stationary process, use the Euclidean
norm for $Z_t$, and let $r\ge1$. If $\widehat\mu$ is fixed and
$L$-Lipschitz, then
\[
 \|V_k-V_k^*\|_r\le(1+L)\|Z_k-Z_k^*\|_r.
\]
If the marginal nonconformity-score density is globally bounded by $M$,
then, with $\delta_r^V(k)=\|V_k-V_k^*\|_r$,
\begin{equation}
 \theta_k\le\min\left\{1,\sqrt{1+2M}\,
                    \delta_r^V(k)^{r/(2r+2)}\right\}.
 \label{eq:general-transfer-bound}
\end{equation}
In particular, geometric nonconformity-score FDM decay implies
$\sum_k\theta_k<\infty$; the bound
$\delta_r^V(k)=O(k^{-a})$ suffices when $a>2+2/r$.
\end{lemma}

\begin{proof}
The reverse triangle inequality gives
$|V_k-V_k^*|\le|Y_k-Y_k^*|+L\|X_k-X_k^*\|
\le(1+L)\|Z_k-Z_k^*\|$. For any $u>0$, a threshold-indicator
flip implies either $|V_k-v|\le u$ or $|V_k-V_k^*|>u$. Thus its
probability is at most $2Mu+\delta_r^V(k)^r/u^r$. If the coefficient
is positive, take $u=\delta_r^V(k)^{r/(r+1)}$, then take square roots
and the supremum over $v$. A zero coefficient is immediate. Summing
the resulting bounds gives the last assertions. No conditional score
density is used. This lemma is a sufficient bound, not an equality of
the dependence of data, scores, and indicators.
\end{proof}

\citet[Proposition~2]{barber2026predictive} also transfer dependence
from data to scores, but use total-variation data processing for
finite-memory score maps. Here we control single-innovation
perturbations and their threshold indicators.

\subsection{Examples and their assumptions}
\label{app:general-examples}

For the Bernoulli recursion in Section~\ref{sec:general}, the stationary
solution is $Y_t=\sum_{j\ge0}2^{-j-1}\varepsilon_{t-j}$.
Its binary digits are independent and fair, so $Y_t$ is uniform on
$[0,1]$ and $V_t=|Y_t-1/2|$ is uniform on $[0,1/2]$.
The recursion also gives $\Cov(Y_0,Y_h)=2^{-h}/12$ for $h\ge0$.
Single-innovation replacement changes $Y_k$, and therefore $V_k$,
by at most $2^{-k-1}$. Since the score density is two, the probability
of a threshold flip is at most $4\cdot2^{-k-1}=2^{1-k}$.
This proves the stated geometric bound on $\theta_k$ and also the
stronger tail condition in Proposition~\ref{prop:variance}.

Let $T(v)=|2v-1/2|$. Direct substitution in the recursion gives
$V_{t-1}=T(V_t)$ almost surely. For every integer $h\ge1$,
the event $A=\{V_0\le1/4\}$ equals
$A_h=\{T^h(V_h)\le1/4\}$ up to a null set. The former is measurable
in the score past and the latter in its future starting at $h$;
both have probability $1/2$. Hence
$|\Pr(A\cap A_h)-\Pr(A)\Pr(A_h)|=1/4$, so the strong-mixing
coefficient does not tend to zero. In particular the scores are not
$\beta$-mixing.

There are nevertheless no inter-time ties. For $h\ge1$, conditional
on the intervening Bernoulli innovations,
$Y_h=2^{-h}Y_0+d_h$ with a fixed $d_h$. The event
$|Y_h-1/2|=|Y_0-1/2|$ forces one of two fixed values of $Y_0$,
because $2^{-h}\ne1$. The variable $Y_0$ is continuous and independent
of these future innovations, so this event has probability zero.
The marginal score density is positive and constant near every
interior quantile. Thus the conditions of Theorem~\ref{thm:marginal}(i),
Theorems~\ref{thm:bahadur}--\ref{thm:realized}, and
Proposition~\ref{prop:variance} hold for this nonmixing example.
Given the innovation past, however, $V_t$ takes the two values
$Y_{t-1}/2$ and $(1-Y_{t-1})/2$, each with probability $1/2$.
These atoms enter any fixed neighborhood of an interior $q$ with
positive probability. The additional local conditional-density
condition in Theorem~\ref{thm:marginal}(ii) therefore fails.

For the raw long-memory example, one concrete choice is
$M_t=\sum_{j\ge0}(j+1)^{-\beta}\zeta_{t-j}$ with
$1/2<\beta<1$ and iid standard normal $\zeta_t$, and independent
iid normal $\eta_t$. The series is well defined in $L^2$ and almost
surely, and its covariance is asymptotic to a positive constant times
$h^{1-2\beta}$, which is nonsummable. The same holds for $Y_t=M_t+\eta_t$
at nonzero lags. Taking innovation pairs $(\zeta_t,\eta_t)$ and
$\widehat\mu(x)=x$ produces $V_t=|\eta_t|$, so
$\theta_k=0$ for $k\ge1$. The example verifies applicability to
long-memory observations: the fixed predictor removes their common
long-memory component, leaving iid nonconformity scores and indicators.

\subsection{Proof of non-asymptotic marginal coverage bounds}
\label{app:general-marginal}

\begin{proof}[Proof of Theorem~\ref{thm:marginal}]
\emph{Part (i).} The sample-size condition ensures $k_n\le n$. Continuity of $F$
implies that for each level in $(0,1)$ there is a finite threshold
attaining it. Put $A=\Theta^2/n$ within this proof. This is positive:
a threshold with $0<F(v)<1$ has positive indicator variance, whereas
\eqref{eq:general-empirical-l2} with $N=1$ bounds that variance by
$\Theta^2$.

For $2/n<u<\min(\alpha,1-\alpha)$ choose $v_-,v_+$ with
$F(v_\pm)=1-\alpha\pm u$. The conformal rank obeys
$0\le k_n/n-(1-\alpha)<2/n$. If $\qhat>v_+$, then
$F_n^{(0)}(v_+)<k_n/n$; if $\qhat<v_-$, then
$F_n^{(0)}(v_-)\ge k_n/n$. Therefore Chebyshev's inequality and
\eqref{eq:general-empirical-l2} imply
\[
 \Pr(\qhat>v_+)\le\frac{A}{(u-2/n)^2},
 \qquad
 \Pr(\qhat<v_-)\le\frac{A}{u^2}.
\]
These implications hold with ties; no identity for
$F_n^{(0)}(\qhat)$ has been used.
For every deterministic $t>n$, the events at $v_-$ and $v_+$ give
\[
 F(v_-)-\Pr(\qhat<v_-)
 \le\Pr(V_t\le\qhat)
 \le F(v_+)+\Pr(\qhat>v_+).
\]
No independence is required. Both one-sided deviations from $1-\alpha$
are consequently at most $u+A/(u-2/n)^2$.

On $u>2/n$, this expression is minimized at
$u=2/n+(2A)^{1/3}$. The sample-size condition makes this choice
admissible, and its value is $2/n+3(A/4)^{1/3}$.
This proves part~(i) and its uniformity over deterministic future
indices. The tolerance $u$ is only a proof device.

\emph{Part (ii).} Again put $A=\Theta^2/n$, and let $\Delta_n=|\qhat-q|$ and
$d_n=k_n/n-(1-\alpha)$. Continuity of $F$ gives $F(q)=1-\alpha$.
The local density lower bound implies
$c_fe\le\min\{\alpha,1-\alpha\}$, so the sample-size condition
ensures $k_n\le n$ and $0\le d_n<2/n\le c_fe/2$.
For $d_n/c_f<u\le e$, the order-statistic inequalities and
\eqref{eq:general-empirical-l2} give
\[
 \Pr(\qhat>q+u)\le\frac{A}{(c_fu-d_n)^2},\qquad
 \Pr(\qhat<q-u)\le\frac{A}{(c_fu+d_n)^2}.
\]
Consequently
\[
 \Pr(\Delta_n>u)
 \le\min\left\{1,\frac{2A}{(c_fu-d_n)^2}\right\},
 \qquad d_n/c_f<u\le e.
\]
Integrating this bound and using the trivial bound one below $d_n/c_f$
yields
\begin{align*}
 \E(\Delta_n\wedge e)
 &\le\frac{d_n}{c_f}
       +\frac1{c_f}\int_0^\infty\min\{1,2A/x^2\}\,dx
   =\frac{d_n+2\sqrt{2A}}{c_f},\\
 \Pr(\Delta_n>e)
 &\le\frac{2A}{(c_fe-d_n)^2}
   \le\frac{8A}{c_f^2e^2}.
\end{align*}

Clip the cutoff to the density neighborhood:
$\widetilde q_n=\max\{q-e,\min\{\qhat,q+e\}\}$.
For a deterministic $t>n$, both cutoffs are $\mathcal F_{t-1}$-measurable.
Let $K_t(v)$ be a version of $\Pr(V_t\le v\mid\mathcal F_{t-1})$.
The conditional-density assumption makes $K_t$ $L$-Lipschitz on
$[q-e,q+e]$ almost surely. Since $\E K_t(q)=1-\alpha$,
\begin{align*}
 |\Pr(V_t\le\qhat)-(1-\alpha)|
 &\le \E|K_t(\qhat)-K_t(\widetilde q_n)|
       +\E|K_t(\widetilde q_n)-K_t(q)|\\
 &\le \Pr(\Delta_n>e)+L\E(\Delta_n\wedge e)\\
 &\le\frac{8\Theta^2}{nc_f^2e^2}
      +\frac{L}{c_f}\left(\frac2n+\frac{2\sqrt2\Theta}{\sqrt n}\right).
\end{align*}
This is \eqref{eq:marginal-bound-smooth} for every such $t$.
Only a truncated cutoff error was integrated; no score moment or
untruncated $L^1$ quantile rate is required.
\end{proof}

\subsection{Local oscillation and the Bahadur representation}
\label{app:general-bahadur}

Under Assumptions~\ref{ass:general}--\ref{ass:quantile}, continuity
and local density positivity give $F(q)=1-\alpha$. Quantile inversion
and \eqref{eq:general-empirical-l2}, applied at $q\pm C/\sqrt n$,
give $\qhat-q=O_{\Pr}(n^{-1/2})$. More explicitly, for each fixed
large $C$ and all sufficiently large $n$, the probability of leaving
this neighborhood is at most $C'\Theta^2/(c_f^2C^2)$; the
$O(n^{-1})$ rank correction is smaller than $c_fC/(2\sqrt n)$.
Taking $C\to\infty$ proves tightness without a score moment condition.

The required random-threshold argument uses
\[
 D_{a,N}(v)=\{F_N^{(a)}(v)-F(v)\}
             -\{F_N^{(a)}(q)-F(q)\},\qquad
 \psi(\delta)=2\sum_{k\ge0}\min\{\sqrt{C_f\delta},\theta_k\}.
\]
For $|v-q|\le\delta\le e$, the increment
$G_t(v)=\ind(V_t\le v)-\ind(V_t\le q)$ has
$\|G_t(v)\|_2\le\sqrt{C_f\delta}$.
Its coupled difference has norm bounded both by
$2\sqrt{C_f\delta}$ and by $2\theta_k$ at lag $k$.
The projection argument of \eqref{eq:general-empirical-l2} therefore
gives $\|D_{a,N}(v)\|_2\le\psi(\delta)/\sqrt N$.
Dominated convergence yields $\psi(\delta)\to0$ as $\delta\downarrow0$.

Partition $[q-\delta,q+\delta]$ into $J$ equal subintervals with
endpoints $v_0,\ldots,v_J$. Monotonicity of both CDFs gives
\[
 \sup_{|v-q|\le\delta}|D_{a,N}(v)|
 \le\max_{0\le j\le J}|D_{a,N}(v_j)|+2C_f\delta/J.
\]
Hence, for every positive $A_N$,
\begin{equation}
 \E\left[A_N^2\sup_{|v-q|\le\delta}|D_{a,N}(v)|^2\right]
 \le2(J+1)\frac{A_N^2}{N}\psi(\delta)^2
      +\frac{8C_f^2(A_N\delta)^2}{J^2}.
 \label{eq:general-local-oscillation}
\end{equation}
If $\delta\to0$, $A_N^2/N$ and $A_N\delta$ are bounded, take
$J\to\infty$ slowly enough that $J\psi(\delta)^2\to0$.
The left side then tends to zero, uniformly in a deterministic block
start. All suprema here are measurable: the empirical CDF has finitely
many jumps and $F$ is continuous, so a countable dense grid suffices.

\begin{proof}[Proof of Theorem~\ref{thm:bahadur}]
Use \eqref{eq:general-local-oscillation} with $N=n$, $A_N=\sqrt n$,
and $\delta=C/\sqrt n$. On $|\qhat-q|\le C/\sqrt n$ it implies
$D_{0,n}(\qhat)=o_{\Pr}(n^{-1/2})$. First take $n\to\infty$,
then $C\to\infty$ using the preceding tightness bound.
No inter-time ties occur almost surely, so the exact rank identity is
$F_n^{(0)}(\qhat)=k_n/n$. Subtracting the empirical CDF at $q$ gives
\[
 F(\qhat)-(1-\alpha)
 =-\{F_n^{(0)}(q)-(1-\alpha)\}
   +\{k_n/n-(1-\alpha)\}-D_{0,n}(\qhat).
\]
The rank correction is $O(n^{-1})$. This proves
\eqref{eq:general-probability-bahadur}; differentiability at $q$,
$f(q)>0$, and $\qhat-q=O_{\Pr}(n^{-1/2})$ prove
\eqref{eq:general-bahadur}.
\end{proof}

\subsection{Martingale approximation and realized coverage}
\label{app:general-clt}

Put $\xi_t=\ind(V_t\le q)-(1-\alpha)$ within the proofs and define
\[
 \mathcal D_t=\sum_{k\ge0}P_t\xi_{t+k}.
\]
This series is absolutely convergent in $L^2$, with
$\|\mathcal D_t\|_2\le\Theta$. The resulting sequence is stationary and
ergodic, since it is a function of the iid innovation shift, and is a
martingale-difference sequence for $(\mathcal F_t)$.
For a block of length $N$, reindexing the projection terms at lag $k$
leaves two boundary sums, each of length $\min(k,N)$. Their terms
are orthogonal, so
\begin{equation}
 \left\|\sum_{t=a+1}^{a+N}(\xi_t-\mathcal D_t)\right\|_2
 \le2\sum_{k\ge1}\sqrt{\min(k,N)}\,\theta_k
 =o(\sqrt N).
 \label{eq:general-martingale-approximation}
\end{equation}
For the last equality, divide by $\sqrt N$ and use dominated
convergence with dominating summable sequence $2\theta_k$.

Orthogonality of martingale projections also gives, for $h\ge0$,
\[
 |\Cov(\xi_0,\xi_h)|
 \le\sum_{j\ge0}\theta_j\theta_{j+h}.
\]
Summing over $h$ proves absolute covariance summability. Thus
$N^{-1}\Var(\sum_{t=1}^N\xi_t)\to\sigma_{\rm cov}^2$.
Comparing this variance with the martingale sum using
\eqref{eq:general-martingale-approximation} shows
$\sigma_{\rm cov}^2=\E(\mathcal D_0)^2$.

\begin{proof}[Proof of Theorem~\ref{thm:realized}]
Set $A_{n,m}=\sqrt{nm/(n+m)}$. For arbitrary $n,m\to\infty$,
$A_{n,m}^2/n\le1$, $A_{n,m}^2/m\le1$, and
$A_{n,m}C/\sqrt n\le C$. Thus
\eqref{eq:general-local-oscillation}, with $N=n,m$,
$A_N=A_{n,m}$, and $\delta=C/\sqrt n$, controls both random-cutoff
terms. The required bounds hold for both block lengths, and
the calibration cutoff lies in this neighborhood with probability
arbitrarily close to one. Subtracting the calibration empirical CDF
from the test empirical CDF at the same cutoff gives
\begin{equation}
 \coverage-(1-\alpha)
 =\frac1m\sum_{t=n+1}^{n+m}\xi_t-\frac1n\sum_{t=1}^n\xi_t
       +o_{\Pr}(A_{n,m}^{-1}).
 \label{eq:general-coverage-expansion}
\end{equation}
The omitted rank correction is $O(n^{-1})=o(A_{n,m}^{-1})$.

For the martingale version of the normalized contrast, use weights
$w_t=-A_{n,m}/n$ on calibration and $w_t=A_{n,m}/m$ on test times.
Their squared sum is one and
$w_*:=\max_t|w_t|\le\max\{n^{-1/2},m^{-1/2}\}\to0$.
Let $Z_t=\E\{(\mathcal D_t)^2\mid\mathcal F_{t-1}\}$.
This is stationary ergodic and integrable. Its calibration average
converges almost surely to $\sigma_{\rm cov}^2$; its test-block average
converges in probability to the same constant by stationarity and the
ergodic theorem. The conditional variance is their convex combination:
\[
 \sum_t w_t^2Z_t
 =\frac{m}{n+m}\frac1n\sum_{t=1}^nZ_t
  +\frac{n}{n+m}\frac1m\sum_{t=n+1}^{n+m}Z_t
 \xrightarrow{\Pr}\sigma_{\rm cov}^2.
\]
The two deterministic weights need not converge.
For every $u>0$, the expectation of the conditional Lindeberg sum is
at most
\[
 \E\left[(\mathcal D_0)^2\ind\{|\mathcal D_0|>u/w_*\}\right]\longrightarrow0.
\]
The martingale central limit theorem \citep{hall1980martingale}
therefore gives the normal limit for $\sum_t w_t\mathcal D_t$.
If $\sigma_{\rm cov}^2=0$, this assertion follows directly in $L^2$.
The two approximation errors in
\eqref{eq:general-martingale-approximation}, multiplied by
$A_{n,m}/n$ and $A_{n,m}/m$, are $o(1)$ in $L^2$.
Indeed $A_{n,m}/\sqrt n\le1$ and $A_{n,m}/\sqrt m\le1$.
Combining this fact with \eqref{eq:general-coverage-expansion} proves
\eqref{eq:general-coverage-clt}. Finite calibration and test blocks
have not been assumed independent.
\end{proof}

\subsection{Block-based variance estimation and studentization}
\label{app:general-variance}

\begin{proof}[Proof of Proposition~\ref{prop:variance}]
Write $\Lambda(k)=\sum_{j\ge k}\theta_j$ and
$\Theta_4=\sum_{k\ge0}\theta_k^{1/2}$.
The tail condition in Proposition~\ref{prop:variance} implies $\Theta_4<\infty$,
since $\theta_k\le\Lambda(k)$ for $k\ge1$.
For the oracle indicators $I_t^o=\ind(V_t\le q)$ let
$S_j=\sum_{t=(j-1)\ell+1}^{j\ell}\{I_t^o-(1-\alpha)\}$ and
$W_j=S_j^2/\ell$.
Denote by $\widehat\sigma_o^2$ the estimator
\eqref{eq:general-batch-variance} with $q$ in place of $\qhat$.
The identity for centering sums of squares gives
\[
 \widehat\sigma_o^2
 =\frac1B\sum_{j=1}^B W_j
   -\ell\left\{\frac1{B\ell}\sum_{t=1}^{B\ell}I_t^o
                    -(1-\alpha)\right\}^2.
\]
The expectation of the last term is at most $\Theta^2/B$ by
\eqref{eq:general-empirical-l2}; it tends to zero. Also
$\E W_j=\Var(\sum_{t=1}^{\ell}\xi_t)/\ell\to\sigma_{\rm cov}^2$.

We next control the variance of the first term. An indicator difference
takes values in $\{-1,0,1\}$, so its $L^4$ norm is the square root of
its $L^2$ norm. Consequently
$\|P_0\xi_k\|_4\le\theta_k^{1/2}$.
The projection moment inequality of
\citet[Theorem~2(i)]{wu2005nonlinear}, at moment order four, yields
\begin{equation}
 \|S_j\|_4\le C\sqrt\ell\,\Theta_4,
 \qquad \|W_j\|_2\le C\Theta_4^2.
 \label{eq:general-batch-fourth}
\end{equation}
All these moments are indicator moments; no moment of $V_t$ is required.

For blocks numbered $1$ and $1+j$, $j\ge2$, replace every innovation
at times at most $\ell$ by an independent copy in the later block.
Write $S_{1+j}^*$ and $W_{1+j}^*$ for the resulting quantities.
They have their original marginal laws and are independent of $W_1$.
Successive single-coordinate replacements give
\[
 \|I_t^o-(I_t^o)^*\|_2
 \le\sum_{s\le\ell}\theta_{t-s}=\Lambda(t-\ell).
\]
For rigor with the infinite past, first replace finitely many coordinates.
The bounds are summable, so these coupled indicators converge in $L^2$;
finite-coordinate approximation of their measurable causal function
identifies the limit with whole-past replacement. Their $L^4$ distance
is the square root of their $L^2$ distance. Thus, for the later block,
\[
 \|S_{1+j}-S_{1+j}^*\|_4
 \le\ell\Lambda((j-1)\ell+1)^{1/2}.
\]
Using $S^2-(S^*)^2=(S-S^*)(S+S^*)$ and
\eqref{eq:general-batch-fourth} gives
\[
 |\Cov(W_1,W_{1+j})|
 \le C\sqrt\ell\,\Lambda((j-1)\ell+1)^{1/2}.
\]
For $j=0,1$, the constant bound from
\eqref{eq:general-batch-fourth} suffices. Since $\Lambda$ is
nonincreasing, disjoint groups of $\ell$ indices give
$\sum_{j\ge1}\Lambda(j\ell+1)^{1/2}
\le\ell^{-1}\sum_{k\ge1}\Lambda(k)^{1/2}$.
Therefore
\[
 \Var\left(\frac1B\sum_{j=1}^B W_j\right)
 \le\frac CB+\frac C{B\sqrt\ell}
                  \sum_{k\ge1}\Lambda(k)^{1/2}\longrightarrow0.
\]
This proves $\widehat\sigma_o^2\to\sigma_{\rm cov}^2$ in probability.

It remains to handle the estimated cutoff. Let
$D_j(v)=D_{(j-1)\ell,\ell}(v)$ and
$\overline D(v)=B^{-1}\sum_{j=1}^B D_j(v)$.
The change in the $j$th block average when replacing $q$ by $\qhat$
is $D_j(\qhat)+F(\qhat)-F(q)$. The last term cancels exactly after
centering the block averages. The reverse triangle inequality for the
Euclidean norm consequently gives
\begin{equation}
 \left|\widehat\sigma_{\rm cov}-\widehat\sigma_o\right|^2
 \le\frac\ell B\sum_{j=1}^B
        \{D_j(\qhat)-\overline D(\qhat)\}^2
 \le\frac\ell B\sum_{j=1}^B D_j(\qhat)^2.
 \label{eq:general-batch-plugin}
\end{equation}
On the event $|\qhat-q|\le C_0/\sqrt n$, bound each summand by its
supremum over this neighborhood. Apply
\eqref{eq:general-local-oscillation} with $N=\ell$,
$A_N=\sqrt\ell$, and $\delta=C_0/\sqrt n$.
Here $A_N^2/N=1$, $A_N\delta=C_0\sqrt{\ell/n}\to0$, and
$\psi(\delta)\to0$. The expectation of the right side of
\eqref{eq:general-batch-plugin}, restricted to this event, tends to
zero by stationarity and the local bound. Markov's inequality, followed
by $C_0\to\infty$ using root-$n$ cutoff tightness, proves that it
is $o_{\Pr}(1)$. Thus the plug-in estimator is consistent.

We finally verify the asserted strict positivity. The tail condition
also implies $\sum_{k\ge1}\Lambda(k)<\infty$. Indeed, eventually
$\Lambda(k)\le\Lambda(k)^{1/2}$. Since
$\|\E(\xi_{t+k}\mid\mathcal F_t)\|_2\le\Lambda(k)$, the series
$H_t=\sum_{k\ge1}\E(\xi_{t+k}\mid\mathcal F_t)$ converges absolutely
in $L^2$. Rearranging these convergent sums yields
$\mathcal D_t=\xi_t+H_t-H_{t-1}$. If
$\sigma_{\rm cov}^2=\E(\mathcal D_0)^2=0$, then
$\xi_t=H_{t-1}-H_t$ almost surely at every integer time.
Put $Z_t=\exp(2\pi\mathrm{i}H_t)$, where $\mathrm{i}^2=-1$.
The integer-valued indicator in $\xi_t$ implies
\[
 Z_t=\exp\{2\pi\mathrm{i}(1-\alpha)\}Z_{t-1}.
\]
In particular $Z_0$ is $\mathcal F_{-r}$-measurable for every
$r\ge1$, since it is a deterministic multiple of $Z_{-r}$.
Triviality of the iid remote-past tail makes $Z_0$ a constant of
modulus one. Stationarity makes that same constant the value of
$Z_{-1}$, contradicting the display because $0<1-\alpha<1$.
Therefore $\sigma_{\rm cov}^2>0$.
\end{proof}

For completeness, Slutsky's theorem now gives the standard normal
limit of the studentized statistic in Section~\ref{sec:general}.
Positive fallback values can be assigned when the estimated variance
is zero; consistency at a positive limit makes such events
asymptotically negligible. Normal continuity at the relevant cutoffs
proves the stated diagnostic size and prediction-interval probability.
Intersecting the interval with $[0,1]$ does not change whether it
contains $\coverage$. Neither scale-estimator independence nor
conditional-on-calibration inference is used.

\paragraph{A fixed-number-of-blocks reference.}
Assume Assumptions~\ref{ass:general}--\ref{ass:quantile} and
$\sigma_{\rm cov}^2>0$. Keep $B\ge2$ fixed, with $\ell\to\infty$,
$B=\lfloor n/\ell\rfloor$, $n=B\ell+o(n)$, and
$m/n\to\kappa\in(0,\infty)$. The block conditions hold, for example,
with $\ell=\lfloor n/B\rfloor$. For the estimator
\eqref{eq:general-batch-variance},
\begin{equation}
 \frac{\coverage-(1-\alpha)}
 {\widehat\sigma_{\rm cov}\sqrt{1/n+1/m}}
 \xrightarrow{d}\sqrt{\frac B{B-1}}\,T_{B-1},
 \qquad T_{B-1}\sim t_{B-1}.
 \label{eq:general-fixed-block-t}
\end{equation}
The tail condition in Proposition~\ref{prop:variance} is one sufficient
condition for the required positivity. Here $\ell/n\to1/B$, giving
a complementary asymptotic regime to that proposition's growing number
of blocks. Equation~\eqref{eq:general-fixed-block-t} motivates the
scaled-$t$ approximation in the main text.

\begin{proof}
Write $\sigma=\sigma_{\rm cov}$ and $A_{n,m}=\sqrt{nm/(n+m)}$
within this proof, and define
\[
 U_{j,n}=\frac1{\sigma\sqrt\ell}
      \sum_{t=(j-1)\ell+1}^{j\ell}\xi_t\quad(1\le j\le B),
 \qquad
 U_{0,n}=\frac1{\sigma\sqrt m}\sum_{t=n+1}^{n+m}\xi_t.
\]
These variables converge jointly to independent standard normals
$(Z_0,Z_1,\ldots,Z_B)$. To see this, apply the martingale
approximation \eqref{eq:general-martingale-approximation} to each
of the finitely many disjoint blocks and use the Cram\'er--Wold device.
For any fixed linear combination, its largest martingale weight tends
to zero, and the conditional variance is a weighted sum of block
averages of $\E\{\mathcal D_t^2\mid\mathcal F_{t-1}\}/\sigma^2$.
Each average converges to one by stationarity and ergodicity.
The $L^2$ Lindeberg argument in the proof of
Theorem~\ref{thm:realized} then gives variance equal to the sum of
the squared combination coefficients. This proves the asserted joint
normal limit and independence.

The incomplete calibration segment has length $n-B\ell=o(n)$.
Its centered indicator sum is $o_{L^2}(\sqrt n)$ by
\eqref{eq:general-empirical-l2}. Writing
$\overline Z=B^{-1}\sum_{j=1}^B Z_j$, the coverage expansion
\eqref{eq:general-coverage-expansion} therefore gives, jointly with
these block variables,
\[
 \frac{A_{n,m}}\sigma\{\coverage-(1-\alpha)\}
 \xrightarrow{d}
 N_\kappa:=\frac{Z_0-\sqrt{\kappa B}\,\overline Z}
                    {\sqrt{1+\kappa}}.
\]

It remains to check the cutoff in the variance estimator. Write
$D_j(v)=D_{(j-1)\ell,\ell}(v)$. The local
bound \eqref{eq:general-local-oscillation}, with $N=\ell$,
$A_N=\sqrt\ell$, and $\delta=C/\sqrt n$, applies because
$A_N^2/N=1$ and $A_N\delta=C\sqrt{\ell/n}$ is bounded.
After localizing the cutoff as before, it gives
$\sqrt\ell D_j(\qhat)=o_{\Pr}(1)$ simultaneously for the fixed
number of blocks. The common shift $F(\qhat)-F(q)$ cancels after
block centering. Hence, jointly with the displayed numerator,
\[
 \frac{\widehat\sigma_{\rm cov}^{\,2}}{\sigma^2}
 \xrightarrow{d}\frac1B\sum_{j=1}^B(Z_j-\overline Z)^2
 =:\frac VB.
\]
Here $V\sim\chi^2_{B-1}$ is independent of $\overline Z$ and $Z_0$.
Thus $N_\kappa\sim N(0,1)$ is independent of $V$. Since $V>0$
almost surely, the continuous mapping theorem gives
$N_\kappa/\sqrt{V/B}=\sqrt{B/(B-1)}\,T_{B-1}$, proving
\eqref{eq:general-fixed-block-t}. This is an asymptotic reference;
the finite-block statistic is not assumed Gaussian or exactly Student-$t$.
\end{proof}

%% file: app_long_memory.tex
\section{Proofs for the Gaussian linear model}
\label{app:long-memory}

This appendix proves Theorems~\ref{thm:lm-bahadur} and~\ref{thm:lm-coverage}
beyond the summable indicator-FDM conditions of Section~\ref{sec:general}.
Lemma~\ref{lem:lm-scale} first establishes the fluctuation scales and
marginal coverage bound; the following calculation explains why the
general FDM assumptions do not apply.
Appendix~\ref{app:lm-bahadur} derives the Bahadur representation and
the calibration--test coverage contrast.
Appendix~\ref{app:lm-laws} uses this contrast to obtain the three
limiting distributions.

Constants in this appendix may depend on the fixed model and on
$\alpha$, but not on the length or deterministic starting point of a
block. The notation $r_N$ means the rate in \eqref{eq:lm-rate}
evaluated at a generic block length $N$.
Set $\gamma(h)=\Cov(W_0,W_h)$ and
$\rho(h)=\gamma(h)/\sigma_W^2$. For a block starting after $a$, write
$F_N^{(a)}(v)=N^{-1}\sum_{t=a+1}^{a+N}\ind\{V_t\le v\}$.
All assertions involving the conformal rank are for sufficiently large
sample sizes, so $k_n\le n$ and
$0\le k_n/n-(1-\alpha)<2/n$.

\begin{lemma}[Square-sum scales and marginal coverage]
\label{lem:lm-scale}
Under Assumption~\ref{ass:lm}, define, for this appendix,
$s_{2,N}^2=\Var\{\sum_{t=1}^N(W_t^2-\sigma_W^2)\}$.
Then
\begin{align}
s_{2,N}^2
&=2\left\{N\sigma_W^4+2\sum_{h=1}^{N-1}(N-h)\gamma(h)^2\right\},
\notag\\
s_{2,N}^2&\sim
\begin{cases}
\displaystyle\frac{4C_\gamma^2}{(3-4\beta)(4-4\beta)}N^{4-4\beta},
 &\beta<3/4,\\
4C_\gamma^2N\log N,&\beta=3/4,\\
\displaystyle 2N\sum_{h\in\mathbb Z}\gamma(h)^2,&\beta>3/4.
\end{cases}
\label{eq:lm-square-scales}
\end{align}
In particular, $s_{2,N}/N\asymp r_N$.
Furthermore $\qhat-q=O_{\Pr}(r_n)$,
$\E|\qhat-q|=O(r_n)$, and
\begin{equation}
\sup_{t>n}|\Pr(V_t\le\qhat)-(1-\alpha)|=O(r_n).
\label{eq:lm-marginal-rate}
\end{equation}
\end{lemma}

\begin{proof}
Square summability of $(a_j)$ gives a stationary Gaussian linear
process. The coefficient bound $|a_j|\le C(1+j)^{-\beta}$ implies
\begin{equation}
\gamma(h)\sim C_\gamma h^{1-2\beta},\qquad
\sum_{j\ge0}|a_ja_{j+h}|\le C(1+h)^{1-2\beta}.
\label{eq:lm-covariance-envelope}
\end{equation}
For the first equivalence, rescale the convolution by $h^{1-2\beta}$.
On any compact subinterval of $(0,\infty)$ the resulting Riemann sum
converges to $c^2x^{-\beta}(1+x)^{-\beta}$; the coefficient bound
controls the discarded neighborhoods of zero and infinity by integrable
power functions. Initial coefficients contribute only $O(h^{-\beta})$.
The Gaussian identity
$\Cov(W_0^2,W_h^2)=2\gamma(h)^2$ now gives
\eqref{eq:lm-square-scales} by summing the covariances. The critical
sum is harmonic; the central sum converges absolutely.

We record a useful covariance bound. If $A$ is a centered even
square-integrable function of a standard normal variable, its Hermite
expansion gives
\begin{equation}
\begin{aligned}
&\Cov\{A(W_0/\sigma_W),A(W_h/\sigma_W)\}
=\sum_{j\ge2,\ j\ {\rm even}}
\frac{\{\E A(Z)H_j(Z)\}^2}{j!}\rho(h)^j,
\\
&0\le\Cov\{A(W_0/\sigma_W),A(W_h/\sigma_W)\}
\le\|A(Z)\|_2^2\rho(h)^2.
\end{aligned}
\label{eq:lm-even-covariance}
\end{equation}
Here $Z\sim N(0,1)$ and $\E H_j(Z)^2=j!$.
The formula follows first for finite Hermite expansions and then by
$L^2$ approximation. It remains valid for negative $\rho(h)$ because
only even orders occur. Both the centered absolute score and the
centered indicator at $q$ have a nonzero second Hermite coefficient.
In fact, with $z=q/\sigma_W$ and $\phi$ the standard normal density,
the leading covariances are
\[
\Cov(V_0,V_h)\sim\frac{\sigma_W^2}{\pi}\rho(h)^2,
\qquad
\Cov\{\ind(V_0\le q),\ind(V_h\le q)\}
\sim2z^2\phi(z)^2\rho(h)^2.
\]
This proves the covariance-memory distinction stated in the main text.
Applied to threshold indicators uniformly in their threshold,
\eqref{eq:lm-even-covariance} also proves
\begin{equation}
\sup_v\Var\{F_n^{(0)}(v)\}\le C r_n^2.
\label{eq:lm-cdf-variance}
\end{equation}

The absolute Gaussian marginal density is smooth and positive near
$q>0$. Write $\Delta_n=\qhat-q$ and fix a sufficiently small $e_0>0$.
Order-statistic inequalities, \eqref{eq:lm-cdf-variance}, and the rank
correction give
$\Pr(|\Delta_n|>u)\le C r_n^2/u^2$ for
$Kr_n\le u\le e_0$, with a fixed $K$.
This proves tightness at scale $r_n$ and contributes $O(r_n)$ to the
integral of the tail over $[0,e_0]$. For $u>e_0$, monotonicity bounds
the tail by $Cr_n^2$. Moreover, the event $\Delta_n>u$ requires at
least $\alpha n/2$ scores above $q+u$ for large $n$, and
$\Delta_n<-u$ requires at least $(1-\alpha)n$ scores below $q-u$.
Markov's inequality on these counts gives
\[
\Pr(|\Delta_n|>u)
\le C_\alpha\Pr(|V_0-q|\ge u)
\le C_\alpha\E(V_0-q)^2/u^2.
\]
Integrating the minimum of this bound and $Cr_n^2$ over
$(e_0,\infty)$ gives $O(r_n)$, hence the asserted $L^1$ rate.

Conditionally on innovations through time $t-1$, $W_t$ is normal
with variance $a_0^2$. Thus the conditional density of $|W_t|$ is
bounded by $2/(|a_0|\sqrt{2\pi})$.
For $t>n$, the cutoff is measurable with respect to this past, so
\[
|\Pr(V_t\le\qhat)-F(q)|
\le\frac{2}{|a_0|\sqrt{2\pi}}\E|\qhat-q|.
\]
This proves \eqref{eq:lm-marginal-rate}, uniformly over deterministic
future indices. The same conditional density shows that for $s<t$,
$\Pr(|W_t|=|W_s|)=0$; consequently there are no inter-time score ties.
\end{proof}

For clarity, summable covariances in the central regime do not imply
the indicator-FDM assumption of Section~\ref{sec:general}.
To verify this for the present model, write
$W_k=U_k+a_k\varepsilon_0$ and
$W_k^*=U_k+a_k\varepsilon_0'$, where $U_k$ is independent of both
innovations and has variance $\sigma_W^2-a_k^2\ge a_0^2$ for $k\ge1$.
Conditioning on the two innovations and bounding the density of $U_k$
gives, uniformly over thresholds $v$,
\[
\Pr\!\left(\ind\{|W_k|\le v\}\ne
\ind\{|W_k^*|\le v\}\right)\le C|a_k|.
\]
For a matching lower bound at $v=q$, restrict
$\varepsilon_0\in[1,2]$ and $\varepsilon_0'\in[-2,-1]$.
For large $k$, $a_k>0$ is small and the two translated intervals have
a symmetric difference of length at least $2a_k$ near $q$.
The Gaussian density of $U_k$ is uniformly bounded below there, so
the disagreement probability is at least $c_0a_k$ for a fixed $c_0>0$.
Thus $\theta_k\asymp |a_k|^{1/2}\asymp k^{-\beta/2}$ and
$\Theta=\infty$ throughout $1/2<\beta<1$.
The Gaussian result in Theorem~\ref{thm:lm-coverage}(i) therefore uses
the predictive-projection argument below, rather than the assumptions
of Theorem~\ref{thm:realized}.

\subsection{Random-cutoff substitution}
\label{app:lm-bahadur}

\begin{proof}[Proof of Theorem~\ref{thm:lm-bahadur}]
Let
\[
D_{a,N}(v)=F_N^{(a)}(v)-F(v)-F_N^{(a)}(q)+F(q).
\]
For $|v-q|\le\delta$ in a fixed neighborhood of $q$, the summand is
a centered even function with variance at most $C\delta$.
Equation~\eqref{eq:lm-even-covariance} therefore gives
$\E D_{a,N}(v)^2\le C\delta r_N^2$.
Put a grid of $J$ equal subintervals on $[q-\delta,q+\delta]$.
Monotonicity of the empirical CDF and local boundedness of $f$ bound
the supremum by the maximum over grid endpoints plus $C\delta/J$.
Bounding a squared maximum by the sum of squares gives
\begin{equation}
\E\sup_{|v-q|\le\delta}|D_{a,N}(v)|^2
\le C(J+1)\delta r_N^2+C\delta^2/J^2.
\label{eq:lm-local-grid}
\end{equation}
Take $\delta=C_0r_n$ for fixed $C_0$, and
$J=\lceil r_n^{-1/2}\rceil$. For calibration $(a,N)=(0,n)$ and
for the adjacent test block $(a,N)=(n,m)$ with $m/n\to\kappa$,
$r_N/r_n$ is bounded. The right side is $o(r_n^2)$.
By Lemma~\ref{lem:lm-scale}, localizing to
$|\qhat-q|\le C_0r_n$ and then letting $C_0\to\infty$ yields
\begin{equation}
D_{0,n}(\qhat)=o_{\Pr}(r_n),\qquad
D_{n,m}(\qhat)=o_{\Pr}(r_n).
\label{eq:lm-random-cutoff}
\end{equation}
For proof-local notation put
$\xi_t=\ind(V_t\le q)-(1-\alpha)$.
The exact identity $F_n^{(0)}(\qhat)=k_n/n$, absence of ties, and
$n^{-1}=o(r_n)$ now give
\begin{equation}
F(\qhat)-(1-\alpha)
=-\frac1n\sum_{t=1}^n\xi_t+o_{\Pr}(r_n).
\label{eq:lm-probability-bahadur}
\end{equation}
Taylor expansion at $q$, with $f(q)>0$, proves
\eqref{eq:lm-bahadur}. The same calculation on the test block gives
the additional representation used below:
\begin{equation}
\coverage-(1-\alpha)
=\frac1m\sum_{t=n+1}^{n+m}\xi_t
-\frac1n\sum_{t=1}^n\xi_t+o_{\Pr}(r_n).
\label{eq:lm-coverage-contrast}
\end{equation}
Only marginal and local empirical bounds were used in
\eqref{eq:lm-random-cutoff}; the random cutoff and test block need
not be independent.
\end{proof}

\subsection{The three limiting distributions}
\label{app:lm-laws}

\begin{proof}[Proof of Theorem~\ref{thm:lm-coverage}]
Write $z=q/\sigma_W>0$. For the centered indicator $\xi_t$,
integration by parts gives the second Hermite coefficient
\[
\E\{\ind(|Z|\le z)H_2(Z)\}=-2z\phi(z).
\]
It is nonzero and the odd coefficients vanish. Consequently
\begin{equation}
\xi_t=-\frac{qf(q)}{2\sigma_W^2}(W_t^2-\sigma_W^2)+A_t,
\label{eq:lm-rank-two}
\end{equation}
where $A_t$ has only even Hermite orders at least four.
Orthogonality and \eqref{eq:lm-covariance-envelope} imply
\begin{equation}
\left\|\sum_{t=1}^N A_t\right\|_2^2
\le C\left\{N+2\sum_{h=1}^{N-1}(N-h)|\rho(h)|^4\right\}
=o(s_{2,N}^2),\qquad \beta\le3/4.
\label{eq:lm-higher-order-remainder}
\end{equation}
Indeed the bracket is $O(N^{6-8\beta})$ for $\beta<5/8$,
$O(N\log N)$ at $\beta=5/8$, and $O(N)$ above it.
Each is negligible relative to \eqref{eq:lm-square-scales} in the
indicated range. Combining \eqref{eq:lm-rank-two} with
\eqref{eq:lm-coverage-contrast} gives
\begin{equation}
\coverage-(1-\alpha)
=\frac{qf(q)}{2\sigma_W^2}
\left\{\frac1n\sum_{t=1}^n(W_t^2-\sigma_W^2)
-\frac1m\sum_{t=n+1}^{n+m}(W_t^2-\sigma_W^2)\right\}
+o_{\Pr}(r_n),\quad \beta\le3/4.
\label{eq:lm-square-contrast}
\end{equation}

If $1/2<\beta<3/4$, the noncentral limit belongs to the classical
Gaussian Hermite theory of \citet{dobrushin1979noncentral,taqqu1975weak}.
The Gaussian rank-two functional limit in
\citet[Eq.~(1.6) and Theorem~1.1]{dehling1989empirical}, applied to
the absolute-score indicator, gives the Rosenblatt limit. Its
covariance hypothesis holds with $D=2\beta-1<1/2$ and slowly varying
$L(h)=h^{2\beta-1}\rho(h)\to C_\gamma/\sigma_W^2>0$;
the source only requires $L$ to be positive eventually.
Using \eqref{eq:lm-rank-two}--\eqref{eq:lm-higher-order-remainder}
and the source's variance-one normalization, the equivalent
adjacent square-sum limit is
\[
\frac1{s_{2,n}}
\left(\sum_{t=1}^n(W_t^2-\sigma_W^2),
\sum_{t=n+1}^{n+m}(W_t^2-\sigma_W^2)\right)
\xrightarrow{d}
\big(R_H(1),R_H(1+\kappa)-R_H(1)\big).
\]
To allow $m/n\to\kappa$, apply the functional theorem on a fixed
compact time interval containing the endpoints; continuity of the
limit permits the varying endpoint. Equations
\eqref{eq:lm-square-scales} and \eqref{eq:lm-square-contrast}
give \eqref{eq:lm-noncentral-law}.

At $\beta=3/4$ the strict-regime source theorem does not apply.
Write $T_N=\sum_{t=1}^N(W_t^2-\sigma_W^2)$.
Stationarity and \eqref{eq:lm-square-scales} give
\[
2\Cov(T_n,T_{n+m}-T_n)=s_{2,n+m}^2-s_{2,n}^2-s_{2,m}^2
=o(n\log n).
\]
The two normalized component variances tend to
$4C_\gamma^2$ and $4C_\gamma^2\kappa$.
For joint normality, every Cram\'er--Wold combination is a quadratic
form $\sum_{t=1}^{n+m}w_{n,t}(W_t^2-\sigma_W^2)$ with
$\max_t|w_{n,t}|=O\{(n\log n)^{-1/2}\}$.
Let $\Gamma_N=(\gamma(i-j))_{i,j\le N}$, $N=n+m$.
The row-sum bound gives $\|\Gamma_N\|_{\rm op}=O(\sqrt N)$.
Hence the eigenvalues $\lambda_{n,j}$ of
$\Gamma_N^{1/2}\operatorname{diag}(w_{n,t})\Gamma_N^{1/2}$ satisfy
$\max_j|\lambda_{n,j}|=O\{(\log n)^{-1/2}\}$.
Diagonalization expresses the form as
$\sum_j\lambda_{n,j}(Z_j^2-1)$ with independent standard normal $Z_j$.
Its variance converges, while
$\sum_j\lambda_{n,j}^4\le
(\max_j\lambda_{n,j}^2)\sum_j\lambda_{n,j}^2\to0$.
Lyapunov's theorem proves convergence to a normal variable for
positive limiting variance; a zero-variance combination converges
in $L^2$ to zero. This proves the required joint Gaussian limit and,
through \eqref{eq:lm-square-contrast}, \eqref{eq:lm-critical-law}.
Only a two-window joint limit, not boundary functional tightness, is
asserted or needed.

Finally suppose $\beta>3/4$.
The rank-two Gaussian limit is consistent with the classical theorem
of \citet{breuer1983central}; we give a predictive-projection argument
that also supplies the coupling used for block sampling.
The full covariance sum of $\xi_t$ is absolutely convergent by
\eqref{eq:lm-even-covariance}, and all its terms are nonnegative.
Thus $\sigma_{\rm cov}^2\ge\alpha(1-\alpha)>0$.
Lemma~\ref{lem:lm-projections} proves
$\sum_{k\ge0}d_k<\infty$, where
$d_k=\|\mathcal P_0\xi_k\|_2$. This is a predictive-projection
condition, distinct from summable indicator FDM.
To make the martingale approximation explicit, put
$D_i=\sum_{j\ge0}\mathcal P_i\xi_{i+j}$, converging in $L^2$.
Then $(D_i)$ is a stationary ergodic square-integrable martingale
difference sequence. Comparing the two boundary terms at each lag
and using orthogonality gives
\[
\frac1{\sqrt N}\left\|
\sum_{t=1}^N\xi_t-\sum_{i=1}^ND_i\right\|_2
\le2\sum_{j\ge0}d_j\min\{\sqrt{j/N},1\}\longrightarrow0.
\]
Dominated convergence proves the limit. The ergodic martingale
central limit theorem gives joint Brownian increments for the two
adjacent sums: ergodicity supplies convergence of the conditional
variance averages, and $\E[D_0^2\ind(|D_0|>\epsilon\sqrt n)]\to0$
supplies Lindeberg's condition. This is the standard
predictive-dependence route; see \citet[Theorem~3(i)]{wu2005nonlinear}.
The $L^2$ approximation and the convergent covariance sum identify
$\E D_0^2=\sigma_{\rm cov}^2$.
Applying \eqref{eq:lm-coverage-contrast} proves
\eqref{eq:lm-central-law}. Higher even orders are not discarded.
These covariance and projection arguments allow arbitrary finitely
many initial coefficients in Assumption~\ref{ass:lm}; their signs need
not be restricted to apply the source results.
\end{proof}

The random-cutoff remainders above are $o_{\Pr}(r_n)$.
We do not infer convergence of the actual normalized coverage variance
from these distributional statements.

%% file: app_blocks.tex
\section{Block sampling: procedure and proof}
\label{app:lm-blocks}

This appendix specifies the block sampling method and proves its validity
with an estimated memory parameter under Theorem~\ref{thm:blocks}'s
conditions. Appendix~\ref{app:lm-procedure} gives the paired-block
procedure and its use for inference, building on
\citet{hall1998sampling,zhang2013block,betken2018subsampling}.
Appendix~\ref{app:lm-exponent} justifies the estimated normalization;
Appendix~\ref{app:lm-block-dependence} separately controls dependence
between coverage blocks. Appendix~\ref{app:lm-block-proof} combines
these ingredients to prove Theorem~\ref{thm:blocks} and justify inference
using the estimated quantiles.

\subsection{Block-sampling procedure}
\label{app:lm-procedure}

\paragraph{Paired calibration and test blocks.}
Choose $b=\lfloor n^\zeta\rfloor$, $0<\zeta<1$, a memory-estimation
block length $h=\lfloor\sqrt n\rfloor$, and
$\ell_b=\lfloor bm/n\rfloor$. For
$k_b=\lceil(b+1)(1-\alpha)\rceil$, take $n$ large enough that
$k_b\le b$, $\ell_b\ge1$, and $\max\{b+\ell_b,2h\}\le n$.
At each start $i=0,\ldots,n-b-\ell_b$, calibrate on
$V_{i+1},\ldots,V_{i+b}$ using rank $k_b$, then evaluate the cutoff
$\widehat q_{i,b}$ on the next $\ell_b$ scores. The block's realized
coverage is
$\widehat{\mathrm{Cov}}_{i,b}=\ell_b^{-1}
\sum_{t=i+b+1}^{i+b+\ell_b}\ind\{V_t\le\widehat q_{i,b}\}$.
All pairs lie within the calibration set and match the target
calibration--test length ratio.

\paragraph{Estimated normalization.}
We estimate memory from the signed residuals $W_t$, whose mean is zero
under Assumption~\ref{ass:lm}, while using $V_t=|W_t|$ for calibration
and evaluation. The two-scale block-variance method of
\citet{zhang2013block}, applied at lengths $l=h,2h$, gives
\begin{equation}
\widehat Q_l=\frac1{n-l+1}\sum_{i=0}^{n-l}
\left(\sum_{t=i+1}^{i+l}W_t\right)^{\!2},
\quad
\widehat\beta=\frac32-
\frac{\log(\widehat Q_{2h}/\widehat Q_h)}{2\log2}.
\label{eq:lm-beta-estimator}
\end{equation}
Use the same normalization formula throughout $1/2<\beta<1$:
\begin{equation}
R_N(u)=\left\{\frac{1+\int_1^N x^{2-4u}\,dx}{N}\right\}^{1/2},
\qquad \widehat r_N=R_N(\widehat\beta),\quad N=b,n.
\label{eq:lm-auto-rate}
\end{equation}
If either estimated block variance is zero, set $\widehat\beta=3/4$.
This convention is asymptotically irrelevant.
Lemma~\ref{lem:lm-exponent} proves that $\widehat r_N/R_N(\beta)\to1$
in probability and that $R_N(\beta)$ is asymptotically proportional
to $r_N$.

\paragraph{Distribution estimation and inference.}
The normalized block CDF is
\begin{equation}
\widehat G_{n,b}(x)=\frac1{n-b-\ell_b+1}
\sum_{i=0}^{n-b-\ell_b}
\ind\!\left\{\frac{\widehat{\mathrm{Cov}}_{i,b}-(1-\alpha)}{\widehat r_b}\le x\right\}.
\label{eq:lm-block-cdf}
\end{equation}
It includes the unknown constants in Theorem~\ref{thm:lm-coverage},
so these need not be estimated separately.
Let $\widehat c_u$ be its $u$-quantile. An asymptotic two-sided
level-$\eta$ test rejects when
$\{\coverage-(1-\alpha)\}/\widehat r_n$ lies outside
$[\widehat c_{\eta/2},\widehat c_{1-\eta/2}]$.
Multiplying these endpoints by $\widehat r_n$ and adding $1-\alpha$
gives a prediction interval for future realized coverage.
The probabilities are over the joint calibration and test sets;
Appendix~\ref{app:lm-block-proof} verifies these claims.

\subsection{Accuracy of the memory estimate and normalization}
\label{app:lm-exponent}

We establish the stronger accuracy
$\widehat\beta-\beta=o_{\Pr}(1/\log n)$, which justifies extrapolating
the normalization from blocks to the full window in all three regimes.

\begin{lemma}[Memory estimation and normalization]
\label{lem:lm-exponent}
Under Assumption~\ref{ass:lm} and the refinement
$a_j=j^{-\beta}\{c+O(j^{-\varphi})\}$ for some $\varphi>0$, the estimator
\eqref{eq:lm-beta-estimator} with $h=\lfloor\sqrt n\rfloor$
satisfies, for any
$0<\delta<\min(\varphi,1-\beta)$,
\begin{equation}
\widehat\beta-\beta
=O_{\Pr}(h^{-\delta}+v_{n,h}),\qquad
v_{n,h}=
\begin{cases}
(h/n)^{2\beta-1},&\beta<3/4,\\
\sqrt{(h/n)\{1+\log(n/h)\}},&\beta=3/4,\\
\sqrt{h/n},&\beta>3/4.
\end{cases}
\label{eq:lm-exponent-rate}
\end{equation}
In particular $(\widehat\beta-\beta)\log n=o_{\Pr}(1)$ and
$\widehat r_N/R_N(\beta)\to1$ in probability for $N=b,n$.
For each fixed $\beta\in(1/2,1)$, $R_N(\beta)/r_N$ converges
to a strictly positive finite constant.
\end{lemma}

\begin{proof}
This applies the two-scale method of \citet[Section~2.5]{zhang2013block}
to the rank-one transform $K(w)=w$, whose self-similarity index is
$3/2-\beta$ throughout $1/2<\beta<1$. We prove the rate directly
for the observed blocks in \eqref{eq:lm-beta-estimator}.

First the coefficient refinement yields
\begin{equation}
\gamma(k)=C_\gamma k^{1-2\beta}\{1+O(k^{-\delta})\},
\qquad 0<\delta<\min(\varphi,1-\beta).
\label{eq:lm-covariance-refinement}
\end{equation}
For the leading coefficients, the sum--integral error in
$\sum_{j\ge1}j^{-\beta}(j+k)^{-\beta}$ is $O(k^{-\beta})$.
Writing the coefficient error as $O(j^{-\beta-\varphi})$,
and splitting its convolution at $j=k$, bounds the remaining error
by $O(k^{1-2\beta-\varphi}+k^{-\beta})$, with an additional
logarithm if $\beta+\varphi=1$.
The strict choice of $\delta$ absorbs this logarithm and all fixed
initial coefficients. Let $S_{i,l}=\sum_{t=i+1}^{i+l}W_t$ and
$s_{1,l}^2=\Var(S_{0,l})$. Summing the covariances gives
\begin{equation}
s_{1,l}^2=K_\beta l^{3-2\beta}\{1+O(l^{-\delta})\},
\qquad K_\beta=\frac{2C_\gamma}{(2-2\beta)(3-2\beta)}.
\label{eq:lm-scale-refinement}
\end{equation}
Indeed the leading sum--integral error is $O(l)$, and the covariance
remainder contributes $O(l^{3-2\beta-\delta})$ because
$\delta<1-\beta<2-2\beta$.

For a lag $k>2l$, the covariance envelope in
\eqref{eq:lm-covariance-envelope} gives
$|\Cov(S_{0,l},S_{k,l})|\le Cl^2k^{1-2\beta}$.
The exact Gaussian covariance-of-squares identity therefore yields
\[
\frac{\Cov(S_{0,l}^2,S_{k,l}^2)}{s_{1,l}^4}
=2\left\{\frac{\Cov(S_{0,l},S_{k,l})}{s_{1,l}^2}\right\}^{\!2}
\le C(l/k)^{4\beta-2}.
\]
For smaller lags the normalized covariance is at most two by
Cauchy--Schwarz. Hence
$A_{n,l}=(n-l+1)^{-1}\sum_{i=0}^{n-l}S_{i,l}^2/s_{1,l}^2$
has expectation one and, for $l=h,2h$,
\begin{equation}
\Var(A_{n,l})\le C
\begin{cases}
(l/n)^{4\beta-2},&\beta<3/4,\\
(l/n)\{1+\log(n/l)\},&\beta=3/4,\\
l/n,&\beta>3/4.
\end{cases}
\label{eq:lm-block-second-moment-rate}
\end{equation}
All blocks in this calculation are observed; no pre-sample blocks
are introduced. Since $\widehat Q_l/s_{1,l}^2=A_{n,l}$,
Chebyshev's inequality gives
$\widehat Q_l/s_{1,l}^2=1+O_{\Pr}(v_{n,h})$ at both lengths.
Combining with \eqref{eq:lm-scale-refinement} gives
$\widehat Q_{2h}/\widehat Q_h
=2^{3-2\beta}\{1+O_{\Pr}(h^{-\delta}+v_{n,h})\}$.
Taking logarithms proves \eqref{eq:lm-exponent-rate}.
Both estimated block variances are positive with probability tending
to one, so the fallback does not affect the conclusion.
For $h=\lfloor\sqrt n\rfloor$, both
$h^{-\delta}\log n$ and $v_{n,h}\log n$ tend to zero.

For the normalization in \eqref{eq:lm-auto-rate}, write
$J_N(u)=NR_N(u)^2$. The integral has the closed form
\[
J_N(u)=
\begin{cases}
1+\{N^{3-4u}-1\}/(3-4u),&u\ne3/4,\\
1+\log N,&u=3/4.
\end{cases}
\]
Near $u=3/4$, evaluate the numerator as
$\operatorname{expm1}\{(3-4u)\log N\}$ to avoid cancellation,
using the continuous limit when $3-4u=0$.
For every real $u$, differentiation under the finite integral gives
\[
|\partial_u\log R_N(u)|
=\frac{2\int_1^N(\log x)x^{2-4u}\,dx}
       {1+\int_1^N x^{2-4u}\,dx}
\le2\log N.
\]
The mean-value theorem and $(\widehat\beta-\beta)\log n=o_{\Pr}(1)$ imply
$\widehat r_N/R_N(\beta)\to1$ for $N=b,n$.
Finally, direct evaluation gives
\begin{equation}
\frac{R_N(\beta)}{r_N}\longrightarrow d_\beta:=
\begin{cases}
\sqrt{1+(4\beta-3)^{-1}},&\beta>3/4,\\
1,&\beta=3/4,\\
(3-4\beta)^{-1/2},&\beta<3/4.
\end{cases}
\label{eq:lm-normalization-constant}
\end{equation}
In particular the unknown constant cancels from the block-to-full
scale ratio:
$\{\widehat r_n/\widehat r_b\}/(r_n/r_b)\to1$ in probability.
These conclusions are pointwise in the fixed model, not uniform over
sequences of memory parameters approaching $3/4$.
\end{proof}

\subsection{Dependence between block statistics}
\label{app:lm-block-dependence}

The memory-estimation lemma controls the normalization, but the
overlapping coverage blocks are themselves dependent.
We show that replacing distant past innovations by independent copies
has a vanishing effect on their normalized statistics. These bounds
will make the variance of the block CDF vanish in
Appendix~\ref{app:lm-block-proof}. They concern coverage statistics,
not estimation of $\beta$.

\paragraph{Quadratic sums for $\beta\le3/4$.}
In these regimes, the leading term of the coverage indicator is
quadratic, so it suffices to control centered-square sums.

Replace all innovations at times at most zero by independent copies,
keeping the innovations at positive times unchanged. A star denotes
the resulting process. For a separation lag $d\ge b$ and
$d<t\le d+b$, write
$W_t=P_t+Q_t$, where $P_t$ contains innovations at times at most zero
and $Q_t$ contains the remaining innovations. Then
$W_t^*=P_t^*+Q_t$, and the Gaussian vectors $P,P^*,Q$ are independent.
Equation~\eqref{eq:lm-covariance-envelope} and the squared coefficient
tail give
\[
|\Cov(P_t,P_s)|\le Cd^{1-2\beta},\qquad
|\Cov(Q_t,Q_s)|\le C(1+|t-s|)^{1-2\beta}.
\]
For $B_{d,b}=\sum_{t=d+1}^{d+b}(W_t^2-\sigma_W^2)$,
the copied difference is
\[
B_{d,b}-B_{d,b}^*
=2\sum_{t=d+1}^{d+b}Q_t(P_t-P_t^*)
+\sum_{t=d+1}^{d+b}\{P_t^2-(P_t^*)^2\}.
\]
Independence and the Gaussian square-covariance identity bound the
squared norms of these two terms by
$Cd^{1-2\beta}b^{3-2\beta}$ and $Cb^2d^{2-4\beta}$, respectively.
Thus
\begin{equation}
\|B_{d,b}-B_{d,b}^*\|_2
\le C\{d^{1/2-\beta}b^{3/2-\beta}+bd^{1-2\beta}\}.
\label{eq:lm-square-copy}
\end{equation}
In the noncentral regime, division by $s_{2,b}\asymp b^{2-2\beta}$
gives $C(b/d)^{\beta-1/2}$.
At the critical boundary, division by $\sqrt{b\log b}$ gives
\begin{equation}
C\frac{(b/d)^{1/4}+(b/d)^{1/2}}{\sqrt{\log b}}.
\label{eq:lm-critical-copy}
\end{equation}

\paragraph{Indicator sums for $\beta>3/4$.}
In this regime, higher Hermite terms also contribute to the limit,
so the full coverage indicator must be controlled.
\label{app:lm-projections}

\begin{lemma}[Predictive projections and copied indicator blocks]
\label{lem:lm-projections}
Let $\xi_t=\ind(|W_t|\le q)-(1-\alpha)$ and
$\mathcal F_j=\sigma(\varepsilon_s:s\le j)$.
For $\mathcal P_j=\E(\cdot\mid\mathcal F_j)
-\E(\cdot\mid\mathcal F_{j-1})$,
\begin{equation}
d_k:=\|\mathcal P_0\xi_k\|_2
\le C|a_k|\left(\sum_{j\ge k}a_j^2\right)^{1/2},\qquad k\ge1.
\label{eq:lm-projection-bound}
\end{equation}
If $\beta>3/4$, then $\sum_{k\ge0}d_k<\infty$.
In this central regime, for whole-past replacement at time zero,
there is a deterministic nonincreasing $e(d)\to0$ such that
\begin{equation}
\left\|\sum_{t=d+1}^{d+b}(\xi_t-\xi_t^*)\right\|_2
\le\sqrt b\,e(d),\qquad d,b\ge1.
\label{eq:lm-central-copy}
\end{equation}
\end{lemma}

\begin{proof}
For $k\ge1$, conditioning on $\mathcal F_0$ integrates out a
Gaussian variable of variance $\sum_{j<k}a_j^2\ge a_0^2$.
The resulting function
$g_k(x)=\Pr(|x+N(0,\sum_{j<k}a_j^2)|\le q)-(1-\alpha)$
is even with $g_k'(0)=0$ and uniformly bounded second derivative.
Put $R_k=\sum_{j>k}a_j\varepsilon_{k-j}$.
Conditional centering in $\varepsilon_0$ and Jensen's inequality bound
$d_k$ by
$\|g_k(R_k+a_k\varepsilon_0)-g_k(R_k+a_k\varepsilon_0')\|_2$.
The mean-value theorem bounds this by
\[
C|a_k|\left\|
|\varepsilon_0-\varepsilon_0'|
\{|R_k|+|a_k|(|\varepsilon_0|+|\varepsilon_0'|)\}
\right\|_2
\le C|a_k|\left(\sum_{j\ge k}a_j^2\right)^{1/2}.
\]
Thus $d_k\le Ck^{1/2-2\beta}$ eventually, proving summability in
the central regime. Finitely many zero or negative coefficients do
not affect the bound.

Let $D_t=\xi_t-\xi_t^*$. In the decomposition $W_t=P_t+Q_t$,
$Q_t$ is independent of $P_t,P_t^*$ and has variance at least $a_0^2$.
The symmetric difference of the two intervals for $Q_t$ has total
length at most $2|P_t-P_t^*|$. Its bounded density therefore gives
\begin{equation}
\|D_t\|_2^2\le C\E|P_t-P_t^*|
\le C\left(\sum_{j\ge t}a_j^2\right)^{1/2},
\qquad
\sup_{t\ge d}\|D_t\|_2\le C d^{(1-2\beta)/4}.
\label{eq:lm-copy-smallness}
\end{equation}
Use the enlarged independent-coordinate filtration with coordinate
$(\varepsilon_j,\varepsilon_j')$ at each $j\le0$ and coordinate
$\varepsilon_j$ at each $j>0$; both past copies enter at the same
time index. Denote its projections by $\mathcal P'_j$.
Each component of $D_t$ has projection norm $d_k$ at lag $k$,
because unused innovations are independent. Orthogonal projection
is also a contraction, so for $t\ge d$,
\[
\|\mathcal P'_{t-k}D_t\|_2
\le\min\{2d_k,Cd^{(1-2\beta)/4}\}.
\]
The remote past of this independent-coordinate filtration is trivial and
$\E D_t=0$, hence $D_t=\sum_{k\ge0}\mathcal P'_{t-k}D_t$ in $L^2$.
For fixed $k$, the projections indexed by $t$ are orthogonal.
Minkowski's inequality consequently gives
\[
\frac1{\sqrt b}\left\|\sum_{t=d+1}^{d+b}D_t\right\|_2
\le\sum_{k\ge0}\min\{2d_k,Cd^{(1-2\beta)/4}\}=:e(d).
\]
Dominated convergence using $\sum_k d_k<\infty$ proves $e(d)\to0$.
This uses predictive projections, not summability of the
single-innovation indicator FDM from the general theory.
\end{proof}

\subsection{Proof of distribution consistency}
\label{app:lm-block-proof}

\begin{proof}[Proof of Theorem~\ref{thm:blocks}]
By \eqref{eq:lm-normalization-constant}, the limit defining $G$ is
the corresponding random variable in Theorem~\ref{thm:lm-coverage}
divided by $d_\beta$.
The CDF $G$ is continuous and nondegenerate in all three regimes.
This is immediate for the two Gaussian limits and their positive
variances. In the noncentral regime, the covariance formula
$\Cov\{R_H(s),R_H(t)\}
=(s^{2H}+t^{2H}-|t-s|^{2H})/2$
shows that the adjacent Rosenblatt contrast has variance
\[
v_H(\kappa)
=\frac{1+\kappa}{\kappa}
\{1+\kappa^{2H-1}-(1+\kappa)^{2H-1}\}>0,
\]
by strict concavity of $x^{2H-1}$.
It is a nonzero second-chaos random variable. Diagonalizing its
Hilbert--Schmidt kernel writes it as an $L^2$-convergent sum
$\sum_j\lambda_j(Z_j^2-1)$. Separate a nonzero term from the
independent remaining sum: the first has no atoms, so their
convolution has no atoms either. This proves continuity of $G$.

Let $L_b=b+\ell_b$, $N_b=n-L_b+1\asymp n$ and, for this proof,
\[
T_{i,b}=\frac1{R_b(\beta)}\left\{
\frac1{\ell_b}\sum_{t=i+b+1}^{i+L_b}\xi_t
-\frac1b\sum_{t=i+1}^{i+b}\xi_t\right\},
\qquad
Y_{i,b}=\frac{\widehat{\mathrm{Cov}}_{i,b}-(1-\alpha)}{R_b(\beta)}.
\]
The joint fixed-threshold limits established in
Appendix~\ref{app:lm-laws}, applied at lengths $b,\ell_b$, imply
$T_{0,b}\Rightarrow G$, since $R_b(\beta)/r_b\to d_\beta$.
Equation~\eqref{eq:lm-coverage-contrast}
at these same lengths gives
\begin{equation}
Y_{0,b}-T_{0,b}=o_{\Pr}(1).
\label{eq:lm-block-cutoff-remainder}
\end{equation}
Stationarity gives the same statements at every deterministic start.
The exact conformal rank $k_b$ is used in each block.

Consider two paired windows with starting points separated by
$d>2L_b$. In the later pair replace every innovation through the end
of the earlier pair. The resulting $T_{d,b}^*$ is independent of
$T_{0,b}$ and has the same marginal law as $T_{d,b}$.
In the noncentral and critical regimes,
\eqref{eq:lm-rank-two}--\eqref{eq:lm-higher-order-remainder}
reduce each indicator sum to its quadratic component with an
$o_{L^2}(br_b)$ remainder. Both subwindow lengths are comparable to
$b$. Applying \eqref{eq:lm-square-copy} or
\eqref{eq:lm-critical-copy} to the two components therefore gives
\begin{equation}
\|T_{d,b}-T_{d,b}^*\|_2
\le e_b+C(b/d)^{\beta-1/2},\qquad e_b\to0,
\quad \beta\le3/4.
\label{eq:lm-oracle-copy-noncentral}
\end{equation}
At the boundary the sharper quadratic term has the additional
factor $1/\sqrt{\log b}$; the displayed weaker bound suffices.
The copied reduction remainder has its original marginal norm.
In the central regime Lemma~\ref{lem:lm-projections}, applied
separately to the two subwindows and shifted to the copying time,
instead gives
\begin{equation}
\|T_{d,b}-T_{d,b}^*\|_2\le C_\kappa e(b),
\qquad e(b)\to0,
\quad \beta>3/4.
\label{eq:lm-oracle-copy-central}
\end{equation}
Thus all three cases have vanishing average copied discrepancy
over well-separated paired windows.

For $I_{i,b}(x)=\ind\{T_{i,b}\le x\}$ and fixed $\epsilon>0$,
copying, independence, and Markov's inequality give
\[
|\Cov(I_{0,b}(x),I_{d,b}(x))|
\le\Pr(|T_{0,b}-x|\le\epsilon)
+\epsilon^{-1}\|T_{d,b}-T_{d,b}^*\|_2.
\]
The $O(b)$ close lags have covariance bounded by one.
For $\beta\le3/4$, averaging
\eqref{eq:lm-oracle-copy-noncentral} over the other lags gives
$e_b+C(b/n)^{\beta-1/2}\to0$; for $\beta>3/4$ use
\eqref{eq:lm-oracle-copy-central}.
Consequently the variance of
$\widetilde G_{n,b}(x)=N_b^{-1}\sum_i I_{i,b}(x)$ has limit superior
at most $2\{G(x+\epsilon)-G(x-\epsilon)\}$.
Letting $\epsilon\downarrow0$ makes it zero, and its expectation
converges to $G(x)$. Pointwise convergence in probability follows.
A finite grid with two tail points and monotonicity of CDFs upgrades
this to uniform convergence over $x\in\mathbb R$.

It remains to replace the oracle contrasts by actual recalibrated
coverages. For each fixed $\epsilon>0$, stationarity and
\eqref{eq:lm-block-cutoff-remainder} imply
\[
\E\left[\frac1{N_b}\sum_i
\ind\{|Y_{i,b}-T_{i,b}|>\epsilon\}\right]
=\Pr(|Y_{0,b}-T_{0,b}|>\epsilon)\longrightarrow0.
\]
Markov's inequality makes the fraction of such blocks negligible.
Sandwiching their empirical CDF between
$\widetilde G_{n,b}(x-\epsilon)$ and
$\widetilde G_{n,b}(x+\epsilon)$, up to this fraction, proves uniform
convergence for the $Y_{i,b}$. This does not require a maximum-over-blocks
remainder bound or an $L^2$ random-cutoff remainder.
Lemma~\ref{lem:lm-exponent} gives
$\widehat r_b/R_b(\beta)\to1$ in all three regimes.
A common random rescaling preserves uniform CDF convergence by
compact-set uniform continuity and tail control.
No independence of the memory estimate and the blocks is needed.
This proves the displayed conclusion of Theorem~\ref{thm:blocks}.
\end{proof}

For the test and interval in Appendix~\ref{app:lm-procedure},
$(\coverage-(1-\alpha))/\widehat r_n\Rightarrow G$, by
Theorem~\ref{thm:lm-coverage} and Lemma~\ref{lem:lm-exponent}.
Uniform CDF convergence places an estimated $u$-quantile between
$G^{-1}(u-\epsilon)$ and $G^{-1}(u+\epsilon)$ with probability tending
to one. Since $G$ is continuous, these deterministic brackets give
asymptotic rejection probability $\eta$ for the two-sided diagnostic,
and coverage probability $1-\eta$ for the future-coverage interval,
by letting $\epsilon\downarrow0$.
This argument also permits flat portions of $G$ and dependent estimated
endpoints. Its probability is over the whole calibration--test
trajectory. The stated null includes the residual-process assumptions
and calibration procedure; a restriction on the mean coverage alone
does not imply the reference law.

%% file: app_non_gaussian.tex
\section{Scope of a non-Gaussian extension}
\label{app:non-gaussian}

This appendix examines what would be needed to extend the Gaussian
results to non-Gaussian innovations; it does not establish an extension.
We first examine the indicator's power rank and a symmetric example,
then discuss empirical-process tools, covariance calculations, and memory
estimation, and finally identify the missing random-cutoff and
critical-boundary arguments. The block results of \citet{zhang2013block}
allow non-Gaussian innovations, but their fixed-transform conclusions
alone do not cover the recalibrated statistics studied here.

For $W_t=\sum_{j\ge0}a_j\varepsilon_{t-j}$ with a smooth marginal
density $f_W$, the first two derivatives of the smoothed absolute
indicator are
\[
\begin{aligned}
\left.\frac{d}{dh}\Pr(|W_0+h|\le q)\right|_{h=0}
&=f_W(-q)-f_W(q),\\
\left.\frac{d^2}{dh^2}\Pr(|W_0+h|\le q)\right|_{h=0}
&=f_W'(q)-f_W'(-q).
\end{aligned}
\]
Symmetric innovations make $W_t$ symmetric and cancel the first
derivative for every nearby cutoff. If $f_W'(q)\ne0$, the absolute
indicator has power rank two. Symmetry alone only implies rank at
least two. Without symmetry the first derivative is generally
nonzero, and a rank-one normalization replaces the quadratic
scale used in Section~\ref{sec:long-memory}.

A concrete class with suitable fixed-transform regularity uses iid
innovations $\varepsilon_t=S_tZ_t/\sqrt{\E S_t^2}$, where $Z_t$ are
standard normal, $S_t$ takes two distinct positive values, and the
two sequences are independent. All moments exist and the innovation
density and its derivatives are smooth and bounded. Conditional on
the scale variables, $W_t$ is centered Gaussian with variance bounded
above and away from zero. Its marginal density is consequently
symmetric and satisfies $f_W'(q)<0$ for every $q>0$.
This verifies the rank-two requirement locally as well as the moment
and smoothing conditions for the fixed-transform block results.

The empirical-process expansion in \citet[Eq.~(4), Theorems~2--3]{wu2003empirical}
suggests a route to the two strict regimes. It requires moment and
smoothness assumptions on an appropriate innovation convolution,
including weighted integrability of its density derivatives.
Subtracting the expansions at the positive and negative endpoints
cancels the first-order term under symmetry. A local reduced-process
bound is also available in \citet[Lemma~15(i), Eq.~(57)]{wu2005bahadur}
under its moment, density, and neighborhood-size conditions.
These tools must be applied to both endpoints and checked uniformly
near the cutoff; a theorem about quantiles of $W_t$ cannot simply be
applied to the nonlinear score $|W_t|$.

The Gaussian covariance identity also changes. For centered iid,
variance-one innovations with finite fourth moment,
\[
\Cov(W_0^2,W_h^2)
=2\Cov(W_0,W_h)^2
+\operatorname{cum}_4(\varepsilon_0)\sum_{j\ge0}a_j^2a_{j+h}^2.
\]
The additional covariance sequence is absolutely summable, since
$\sum_j a_j^2<\infty$, and contributes $O(n)$ to a square partial-sum
variance. It is negligible relative to $n^{4-4\beta}$ for
$\beta<3/4$, but may change central variances. In the noncentral
rank-two reduction
the indicator coefficient is $f_W'(q)$ rather than the
Gaussian-specific $-qf(q)/(2\sigma_W^2)$. Actual recalibrated coverage
blocks would include that coefficient automatically; a separate
density-derivative estimator is not intrinsic to the method.

For the all-moment mixture class above, Condition~1 with $\nu=4$ of
\citet{zhang2013block} holds for the linear transform $K(w)=w$.
The coefficient refinement
$a_j=j^{-\beta}\{c+O(j^{-\varphi})\}$, $\varphi>0$, supplies their
Condition~2; it is not a consequence of the innovation distribution.
Their Theorem~3 and Corollary~1 then support estimation of the memory
parameter from signed residuals throughout $1/2<\beta<1$.
This step concerns the linear residuals, not the rank-two coverage
statistic. The Gaussian block-moment calculation in
Lemma~\ref{lem:lm-exponent} cannot simply be imported unchanged;
the cited non-Gaussian moment conditions provide the relevant
alternative. Neither memory estimation nor a consistent scale alone
establishes a limiting law for actual recalibrated block coverages.

Completing a non-Gaussian extension still requires a local
random-cutoff argument and its transfer to actual block coverages.
The Gaussian even-Hermite covariance bound used in
\eqref{eq:lm-local-grid} is not available unchanged. The non-Gaussian
critical boundary also needs a separate proof: the cited strict-regime
theorems do not include it, and the Gaussian quadratic-form argument
above does not transfer directly. Thus this discussion identifies a
concrete extension route, not an established extension of
Theorems~\ref{thm:lm-bahadur}--\ref{thm:blocks}.

%% file: experiments_appendix.tex
\section{Additional simulation details}
\label{app:simulation-details}

This appendix provides the assumption checks, implementation details,
and additional comparisons for Section~\ref{sec:experiments}.
Appendix~\ref{app:study1-details} verifies the nonmixing example and its
FDM bound. Appendix~\ref{app:study2-details} specifies the nonlinear
models and the exchangeable and corrected $t$ comparisons.
Appendix~\ref{app:study3-details} documents the long-memory experiment:
Appendix~\ref{app:study3-setup} gives its setup and memory-estimation results,
Appendix~\ref{app:study3-benchmarks} compares block sampling with moving-block
bootstrap and HAC inference, Appendix~\ref{app:study3-extra} gives the
additional $\beta=0.7$ results.

\subsection{Study 1: marginal coverage}
\label{app:study1-details}
We use 20,000 independent replications for each sample size and nominal
coverage level.
The stationary state $U_t$ is uniform on $[0,1]$. A score-threshold
event $|U_t^2-1/3|\le v$ is an interval in $U_t$. Replacing one
Bernoulli innovation changes $U_k$ by at most $2^{-k-1}$; crossing
either endpoint therefore gives
$\theta_k\le\min\{1,2^{(1-k)/2}\}$. The score CDF is continuous,
so the assumptions of Theorem~\ref{thm:marginal}(i) hold.
In contrast, whenever $V_j>1/3$, the state is recovered as
$U_j=\sqrt{V_j+1/3}$. Such events occur infinitely often, and
$U_{t-1}=2U_t\bmod1$ recovers $U_0$ from every future score tail.
A probability-$1/2$ event in $V_0$ belongs to both the score past and
every future tail, so the strong-mixing coefficients are at least
$1/4$. In particular, the model is not $\beta$-mixing.

The implementation uses a 128-bit binary state and exact integer score
comparisons. It approximates the infinite-past model; the score error
under the binary coupling is less than $2^{1-128}$. Nonmixing is a
property of the mathematical model, not a conclusion from finite-precision
simulation. If $k_n>n$, we use an infinite conformal cutoff.

\subsection{Study 2: realized-coverage inference for general time series}
\label{app:study2-details}
We use 10,000 independent trajectories for each model and sample size.
We consider two nonlinear models: ARCH(1),
$Y_t=\sqrt{0.5+0.5Y_{t-1}^2}\,\varepsilon_t$,
and stochastic volatility (SV),
$Y_t=e^{H_t/2-1/4}\varepsilon_t$ with
$H_t=0.9H_{t-1}+\sqrt{0.19}\,\eta_t$.
The innovations are independent standard normal variables, and the
score is $V_t=|Y_t|$.
ARCH trajectories start from zero and discard 2,000 observations;
SV trajectories start with an independent $H_0\sim N(0,1)$.
The respective indicator FDM bounds are $\theta_k\lesssim0.5^{k/6}$
and $\theta_k\lesssim0.9^{k/3}$, so both models meet the stronger
dependence condition in Proposition~\ref{prop:variance}.
The score densities are positive and continuous near the target
quantiles, and inter-time ties have probability zero.
Only complete calibration blocks enter the variance estimator;
the full calibration set of length $n$ determines the cutoff.
The QQ plots use quantile levels $0.01,0.02,\ldots,0.99$.
Both test levels use the same saved studentized statistics and
predictive miscoverage $\alpha=0.1$.

Both designs also meet the local conditional-density condition in
Theorem~\ref{thm:marginal}(ii). For ARCH the conditional scale is at
least $1/\sqrt2$, giving the folded-normal density bound $2/\sqrt\pi$.
For SV, for every $v>0$,
$\sup_{s>0}\sqrt{2/\pi}s^{-1}\exp\{-v^2/(2s^2)\}
=\sqrt{2/(\pi\mathrm e)}/v$.
Conditional mixing over scales preserves this bound on a neighborhood
of the positive target quantile. The marginal densities have a positive
local lower bound there.

\paragraph{Comparison with exchangeable and corrected $t$ references.}
\label{app:study2-comparisons}
Table~\ref{tab:study2-comparisons} compares references on the same trajectories.
The exchangeable reference divides the coverage error by
$\sqrt{\alpha(1-\alpha)(1/n+1/m)}$ and uses normal critical values.
With $B=\lfloor n/\ell\rfloor$ complete calibration blocks, the corrected
$t$ reference uses the same studentized statistic but replaces the normal
cutoff by $\sqrt{B/(B-1)}\,t_{B-1,1-\eta/2}$.
This allows for variance-estimation uncertainty with few blocks; the scale
factor accounts for divisor $B$ rather than $B-1$.
It is motivated by the fixed-$B$ asymptotic reference in
Appendix~\ref{app:general-variance}. No statistic is recentered or
empirically rescaled. At $n=2^{19}$, the exchangeable nominal-95\%
reference intervals contain realized coverage in only 83.96\% and
67.90\% of replications for ARCH and SV, respectively.

\begin{table}[htbp]
\caption{For Study 2, rejection rates (\%) under three references, using
the same saved trajectories. Columns specify the nominal test level.}
\label{tab:study2-comparisons}
\begin{center}
\small
\setlength{\tabcolsep}{4pt}
\input{experiment_assets/study2_comparisons_20260923/comparison_table.tex}
\end{center}
\end{table}

\subsection{Study 3: block sampling method under long memory}
\label{app:study3-details}
\subsubsection{Simulation setup}
\label{app:study3-setup}
For each $\beta$ and sample size, we use 10,000 independent trajectories
to evaluate the procedure and a separate sample of 10,000 independent
trajectories to estimate the reference distribution.
Stationary Gaussian trajectories are generated by FFT circulant
embedding using numerically evaluated covariances
$\gamma_\beta(k)=\sum_{j\ge0}(1+j)^{-\beta}(1+j+k)^{-\beta}$.
Thus the simulation approximates the infinite linear model through its
covariance function, rather than truncating each trajectory's past.
We use the procedure in Appendix~\ref{app:lm-procedure}, with signed
residuals $W_t=Y_t$, absolute scores $V_t=|Y_t|$, pilot length
$h=\lfloor\sqrt n\rfloor$, and coverage-block length
$b=\lfloor n^{2/3}\rfloor$. The $n^{2/3}$ choice also appears in
long-memory covariance inference \citep{zhai2026simultaneous}; longer
blocks capture dependence over longer spans while $b/n\to0$.
These choices of $h$ and $b$ are simple and satisfy our asymptotic conditions.
Since the residual model has mean zero,
we use uncentered block second moments to avoid the additional
estimation error associated with the slowly converging sample mean
under long memory.
The estimate $\widehat\beta$ uses only the calibration set and is
not clipped. It determines the same continuous normalization at
$N=b,n$; the true memory parameter is not supplied to the procedure.
We use equal adjacent calibration and test windows, $m=n$.
Every historical pair has lengths $b,b$ and uses its own exact conformal
rank; all $n-2b+1$ pairs lie within the calibration set.

The independent reference sample approximates the finite-$n$
distribution, not an imposed limiting Gaussian distribution.
For each calibration history we multiply its raw block errors by
$R_n(\widehat\beta)/R_b(\widehat\beta)$ and compare the resulting CDF with the
independent CDF of full-window realized-coverage errors.
No estimated normalization enters the reference CDF.
Table~\ref{tab:study3-benchmark-ks} reports the median of
this Kolmogorov distance over calibration histories.
The QQ plots use levels
$0.025,0.05,0.10,0.20,\ldots,0.90,0.95,0.975$.
The two-sided level-$\eta$ diagnostic rejects when the observed
full-window error lies strictly outside the scaled block quantiles at
$\eta/2$ and $1-\eta/2$.
Nominal prediction coverage remains $1-\alpha=0.9$ at both test levels.
Table~\ref{tab:study3-estimation} summarizes the
bias and RMSE of $\widehat\beta$.
No trajectories were discarded and no normalization failures occurred.

\begin{table}[htbp]
\caption{For Study 3, bias and root mean squared error (RMSE) of
$\widehat\beta$.
All entries are dimensionless.}
\label{tab:study3-estimation}
\begin{center}
\small
\setlength{\tabcolsep}{2.6pt}
\input{experiment_assets/study3_compact_20260925/beta_estimation.tex}
\end{center}
\end{table}

\subsubsection{Comparison with moving-block bootstrap and HAC inference}
\label{app:study3-benchmarks}
We compare block sampling with moving-block bootstrap (MBB)
\citep{kunsch1989jackknife} and Gaussian inference using a
heteroskedasticity and autocorrelation consistent (HAC) variance estimator
\citep{newey1987simple}. The comparison uses the same trajectories
in every setting with
$\beta\in\{0.9,0.75,0.6\}$ and $n\in\{2^{12},2^{14},2^{16},2^{18},2^{20}\}$,
with $m=n$. All methods estimate coverage uncertainty from the calibration
set only, and use the same independent reference sample.
Our block sampling procedure is unchanged and estimates $\beta$;
neither benchmark uses the true memory parameter or its regime.
The MBB block length and HAC bandwidth both equal the coverage-block
length, $L=b=\lfloor n^{2/3}\rfloor$.

MBB samples, independently with replacement, from the $n-L+1$ overlapping
length-$L$ blocks of calibration scores. Concatenating sampled blocks and
truncating at $2n$ gives one pseudo trajectory. We recompute the exact
conformal cutoff from its first $n$ scores and evaluate coverage on the
next $n$, allowing a sampled block to cross this boundary. Repeating this
999 times estimates the distribution of bootstrap realized coverage minus
$1-\alpha$. Bootstrap ties use the same non-strict coverage inequality as
the original procedure, and empirical quantiles use the left inverse CDF.

For HAC, let $I_t=\ind\{V_t\le\qhat\}$ and
$\bar I=n^{-1}\sum_{t=1}^n I_t$. We compute
$\widehat\gamma_k=n^{-1}\sum_{t=1}^{n-k}(I_t-\bar I)(I_{t+k}-\bar I)$
and the Bartlett estimator
$\widehat v_L=\widehat\gamma_0+
2\sum_{k=1}^L\{1-k/(L+1)\}\widehat\gamma_k$.
The reference distribution is $N(0,\widehat v_L(1/n+1/m))$,
which includes calibration uncertainty.

Table~\ref{tab:study3} reports 5\% rejection rates at
$n=2^{12},2^{16},2^{20}$; Table~\ref{tab:study3-rejection} gives the
remaining settings. Table~\ref{tab:study3-benchmark10} gives 10\% rates
over the full grid. The additional $\beta=0.7$ results use block sampling
only; dashes mark benchmarks not run in this setting.
The Monte Carlo unit is an independent trajectory, not an overlapping
block or bootstrap draw. No trajectory failed and no HAC variance was zero.

\begin{table}[htbp]
\caption{For Study 3, supplementary 5\% rejection rates (\%).
The upper panel completes Table~\ref{tab:study3}; the lower panel gives
the additional $\beta=0.7$ block sampling results.}
\label{tab:study3-rejection}
\begin{center}
\small
\setlength{\tabcolsep}{4pt}
\input{experiment_assets/study3_benchmarks_10k_20260925/benchmark_rejection_05_supplement.tex}
\end{center}
\end{table}

\begin{table}[htbp]
\caption{For Study 3, rejection rates (\%) at the 10\% test level.
BS denotes block sampling; all methods use length
$b=\lfloor n^{2/3}\rfloor$.}
\label{tab:study3-benchmark10}
\begin{center}
\small
\setlength{\tabcolsep}{2pt}
\input{experiment_assets/study3_benchmarks_10k_20260925/benchmark_rejection_10.tex}
\end{center}
\end{table}

Figure~\ref{fig:study3-benchmark-qq} compares estimated coverage-error
quantiles with the independent reference. Under stronger memory, the
benchmark quantiles are too concentrated near zero, whereas block sampling
better captures their spread. Tables~\ref{tab:study3-benchmark-ks}
and~\ref{tab:study3-benchmark-width} quantify this comparison through
Kolmogorov distances and interval widths across all five sample sizes.
At $\beta=0.6$ and $n=2^{20}$,
the 10\% rejection rate is 10.34\% for block sampling, compared with
51.78\% for MBB and 51.16\% for HAC.
Thus the benchmark methods substantially underestimate coverage uncertainty
under strong memory, while block sampling substantially improves test calibration.

\begin{figure}[htbp]
\centering
\includegraphics[width=\linewidth]{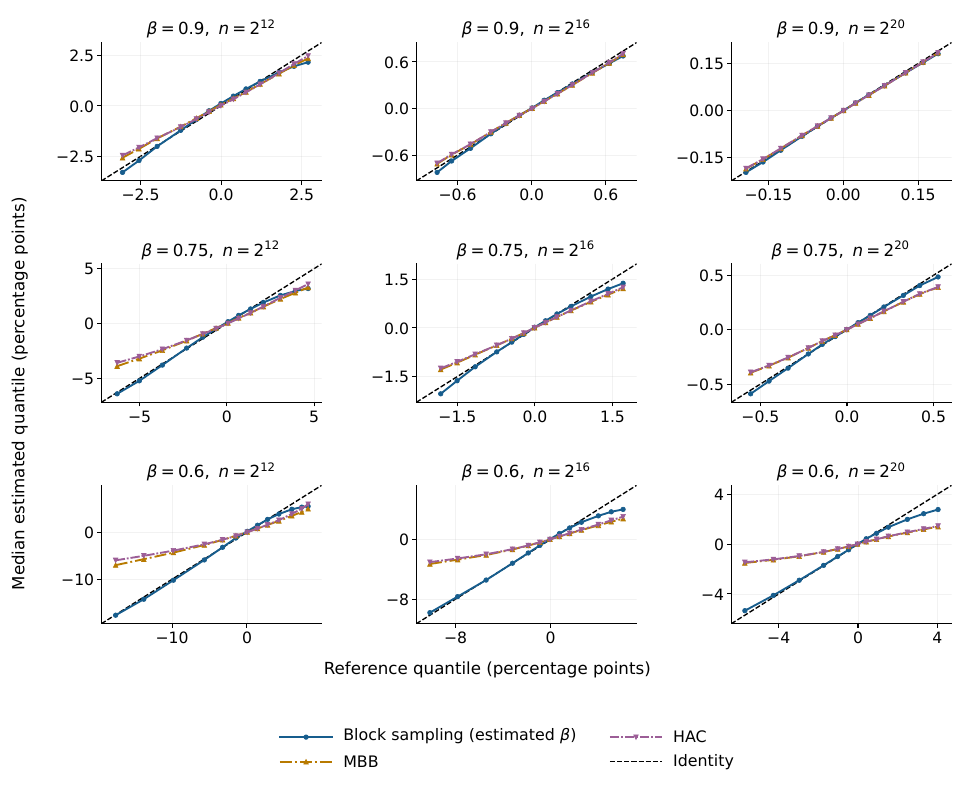}
\caption{For Study 3, comparative QQ plots of realized-coverage errors
(percentage points). Each curve gives the median estimated quantile across
paired histories, against the corresponding quantile of the
independent reference sample. Closeness to the dashed
equality line indicates agreement with the finite-$n$ coverage-error
distribution. Rows give $\beta$ and columns give $n$.}
\label{fig:study3-benchmark-qq}
\end{figure}

\begin{table}[htbp]
\caption{For Study 3, median Kolmogorov distance from the independent
finite-$n$ reference CDF over simulation histories.
BS denotes block sampling. Smaller values indicate closer distributional
agreement.}
\label{tab:study3-benchmark-ks}
\begin{center}
\small
\setlength{\tabcolsep}{2pt}
\input{experiment_assets/study3_benchmarks_10k_20260925/benchmark_median_ks.tex}
\end{center}
\end{table}

\begin{table}[tp]
\caption{For Study 3, median estimated 95\% coverage-error interval widths
over paired histories, in percentage points.
BS denotes block sampling; Ref. is the 97.5th minus 2.5th quantile of the
independent reference sample. Widths below Ref. indicate
underestimation of coverage uncertainty.}
\label{tab:study3-benchmark-width}
\begin{center}
\small
\setlength{\tabcolsep}{3pt}
\input{experiment_assets/study3_benchmarks_10k_20260925/benchmark_width.tex}
\end{center}
\end{table}

\subsubsection{Additional ultra-long-memory results}
\label{app:study3-extra}
We also study $\beta=0.7$, another ultra-long-memory model with a
non-Gaussian limit. Its coverage-error scale is $n^{-0.4}$, so its
memory is less persistent than that of the $\beta=0.6$ example in the
main text. The procedure and sample sizes are unchanged.
Figure~\ref{fig:study3-beta07} and Tables~\ref{tab:study3-rejection}
and~\ref{tab:study3-benchmark10}
give the full sample-size comparison. The median Kolmogorov distance
decreases from 0.1467 to 0.0572; at $n=2^{20}$, the 5\% and 10\%
rejection rates are 4.57\% and 8.37\%, respectively, both below nominal.
Distributional approximation improves with $n$ for
both ultra-long-memory examples, but more slowly for $\beta=0.6$.

\begin{figure}[htbp]
\centering
\includegraphics[width=\linewidth]{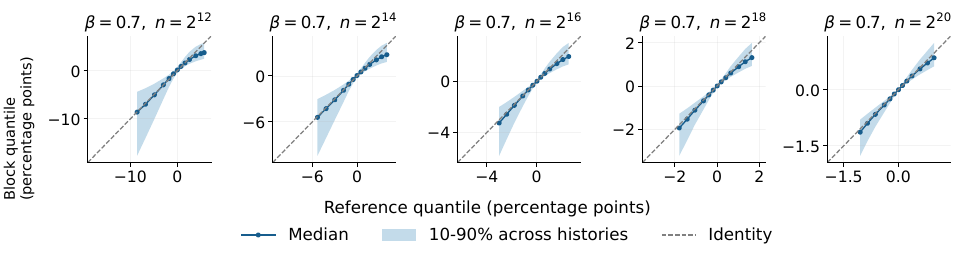}
\caption{For Study 3, block versus independent finite-$n$
reference error quantiles at $\beta=0.7$ (percentage points). Curves and shading have the same
interpretation as in Figure~\ref{fig:study3}.}
\label{fig:study3-beta07}
\end{figure}

%% file: experiment_assets/study2_comparisons_20260923/comparison_table.tex
\begin{tabular}{llrrrrrr}
\toprule
& & \multicolumn{2}{c}{Normal} & \multicolumn{2}{c}{Exchangeable SE} & \multicolumn{2}{c}{Corrected $t$} \\
Model & $n$ & 5\% & 10\% & 5\% & 10\% & 5\% & 10\% \\
\midrule
ARCH & $2^{9}$ & 10.54 & 16.27 & 14.38 & 23.19 & 4.75 & 9.60 \\
ARCH & $2^{11}$ & 9.12 & 14.67 & 15.86 & 24.20 & 5.39 & 10.37 \\
ARCH & $2^{13}$ & 7.48 & 12.79 & 15.77 & 23.39 & 5.31 & 10.34 \\
ARCH & $2^{15}$ & 6.40 & 11.81 & 16.15 & 24.21 & 5.08 & 9.95 \\
ARCH & $2^{17}$ & 6.02 & 11.50 & 16.14 & 23.69 & 5.29 & 10.45 \\
ARCH & $2^{19}$ & 5.49 & 10.77 & 16.04 & 24.26 & 5.03 & 10.11 \\
\midrule
SV & $2^{9}$ & 12.92 & 19.65 & 30.57 & 40.59 & 6.19 & 11.89 \\
SV & $2^{11}$ & 8.98 & 14.87 & 31.63 & 40.42 & 5.37 & 10.28 \\
SV & $2^{13}$ & 7.60 & 13.20 & 32.58 & 41.45 & 5.36 & 10.74 \\
SV & $2^{15}$ & 6.03 & 11.25 & 31.45 & 39.71 & 4.89 & 9.86 \\
SV & $2^{17}$ & 5.75 & 11.04 & 32.10 & 40.56 & 4.79 & 9.94 \\
SV & $2^{19}$ & 5.61 & 10.94 & 32.10 & 39.96 & 5.15 & 10.24 \\
\bottomrule
\end{tabular}

%% file: experiment_assets/study3_compact_20260925/beta_estimation.tex
\begin{tabular}{crrrrrrrrrr}
\toprule
$n=m$ & \multicolumn{2}{c}{$2^{12}$} & \multicolumn{2}{c}{$2^{14}$} & \multicolumn{2}{c}{$2^{16}$} & \multicolumn{2}{c}{$2^{18}$} & \multicolumn{2}{c}{$2^{20}$} \\
\cmidrule(lr){2-3}\cmidrule(lr){4-5}\cmidrule(lr){6-7}\cmidrule(lr){8-9}\cmidrule(lr){10-11}
$\beta$ & Bias & RMSE & Bias & RMSE & Bias & RMSE & Bias & RMSE & Bias & RMSE \\
\midrule
0.9 & -0.1283 & 0.1452 & -0.1177 & 0.1275 & -0.1060 & 0.1116 & -0.0943 & 0.0975 & -0.0840 & 0.0859 \\
0.75 & -0.0542 & 0.0849 & -0.0497 & 0.0695 & -0.0436 & 0.0569 & -0.0381 & 0.0472 & -0.0325 & 0.0384 \\
0.7 & -0.0304 & 0.0710 & -0.0296 & 0.0572 & -0.0271 & 0.0466 & -0.0236 & 0.0377 & -0.0200 & 0.0307 \\
0.6 & 0.0099 & 0.0607 & 0.0070 & 0.0490 & 0.0045 & 0.0398 & 0.0031 & 0.0338 & 0.0024 & 0.0291 \\
\bottomrule
\end{tabular}

%% file: experiment_assets/study3_benchmarks_10k_20260925/benchmark_rejection_05_supplement.tex
\begin{tabular}{crrrrrr}
\toprule
$n=m$ & \multicolumn{3}{c}{$2^{14}$} & \multicolumn{3}{c}{$2^{18}$} \\
\cmidrule(lr){2-4}\cmidrule(lr){5-7}
$\beta$ & BS & MBB & HAC & BS & MBB & HAC \\
\midrule
0.9 & 6.55 & 8.26 & 8.14 & 5.29 & 6.07 & 6.20 \\
0.75 & 7.24 & 15.14 & 14.68 & 4.93 & 14.31 & 14.35 \\
0.6 & 9.94 & 34.38 & 31.45 & 6.88 & 40.80 & 39.88 \\
\bottomrule
\end{tabular}
\par\smallskip
\begin{tabular}{crrrrr}
\toprule
$n=m$ & $2^{12}$ & $2^{14}$ & $2^{16}$ & $2^{18}$ & $2^{20}$ \\
\midrule
$\beta=0.7$, BS & 10.68 & 7.83 & 5.79 & 5.20 & 4.57 \\
\bottomrule
\end{tabular}

%% file: experiment_assets/study3_benchmarks_10k_20260925/benchmark_rejection_10.tex
\begin{tabular}{crrrrrrrrrrrrrrr}
\toprule
$n=m$ & \multicolumn{3}{c}{$2^{12}$} & \multicolumn{3}{c}{$2^{14}$} & \multicolumn{3}{c}{$2^{16}$} & \multicolumn{3}{c}{$2^{18}$} & \multicolumn{3}{c}{$2^{20}$} \\
\cmidrule(lr){2-4}\cmidrule(lr){5-7}\cmidrule(lr){8-10}\cmidrule(lr){11-13}\cmidrule(lr){14-16}
$\beta$ & BS & MBB & HAC & BS & MBB & HAC & BS & MBB & HAC & BS & MBB & HAC & BS & MBB & HAC \\
\midrule
0.9 & 13.40 & 17.13 & 16.82 & 11.16 & 14.26 & 14.16 & 10.14 & 12.07 & 12.14 & 10.19 & 11.28 & 11.36 & 10.12 & 10.88 & 10.64 \\
0.75 & 14.22 & 24.45 & 22.81 & 11.29 & 22.27 & 21.63 & 10.60 & 21.91 & 21.64 & 9.78 & 21.37 & 21.37 & 8.56 & 19.80 & 19.82 \\
0.7 & 14.44 & --- & --- & 11.99 & --- & --- & 10.26 & --- & --- & 9.46 & --- & --- & 8.37 & --- & --- \\
0.6 & 17.18 & 39.81 & 36.17 & 14.35 & 42.17 & 39.67 & 12.48 & 45.26 & 43.96 & 11.32 & 48.22 & 47.72 & 10.34 & 51.78 & 51.16 \\
\bottomrule
\end{tabular}

%% file: experiment_assets/study3_benchmarks_10k_20260925/benchmark_median_ks.tex
\begin{tabular}{crrrrrrrrrrrrrrr}
\toprule
$n=m$ & \multicolumn{3}{c}{$2^{12}$} & \multicolumn{3}{c}{$2^{14}$} & \multicolumn{3}{c}{$2^{16}$} & \multicolumn{3}{c}{$2^{18}$} & \multicolumn{3}{c}{$2^{20}$} \\
\cmidrule(lr){2-4}\cmidrule(lr){5-7}\cmidrule(lr){8-10}\cmidrule(lr){11-13}\cmidrule(lr){14-16}
$\beta$ & BS & MBB & HAC & BS & MBB & HAC & BS & MBB & HAC & BS & MBB & HAC & BS & MBB & HAC \\
\midrule
0.9 & 0.108 & 0.060 & 0.054 & 0.074 & 0.048 & 0.036 & 0.052 & 0.043 & 0.027 & 0.036 & 0.036 & 0.021 & 0.027 & 0.034 & 0.015 \\
0.75 & 0.133 & 0.098 & 0.100 & 0.103 & 0.088 & 0.076 & 0.077 & 0.085 & 0.079 & 0.060 & 0.082 & 0.074 & 0.045 & 0.078 & 0.069 \\
0.7 & 0.147 & --- & --- & 0.114 & --- & --- & 0.088 & --- & --- & 0.075 & --- & --- & 0.057 & --- & --- \\
0.6 & 0.168 & 0.164 & 0.178 & 0.139 & 0.170 & 0.175 & 0.119 & 0.190 & 0.195 & 0.099 & 0.203 & 0.207 & 0.093 & 0.220 & 0.220 \\
\bottomrule
\end{tabular}

%% file: experiment_assets/study3_benchmarks_10k_20260925/benchmark_width.tex
\begin{tabular}{crrrrrrrrrrrr}
\toprule
$n=m$ & \multicolumn{4}{c}{$\beta=0.9$} & \multicolumn{4}{c}{$\beta=0.75$} & \multicolumn{4}{c}{$\beta=0.6$} \\
\cmidrule(lr){2-5}\cmidrule(lr){6-9}\cmidrule(lr){10-13}
 & BS & MBB & HAC & Ref. & BS & MBB & HAC & Ref. & BS & MBB & HAC & Ref. \\
\midrule
$2^{12}$ & 5.453 & 4.907 & 4.915 & 5.762 & 9.663 & 7.178 & 7.175 & 10.938 & 23.254 & 11.987 & 11.941 & 25.708 \\
$2^{14}$ & 2.928 & 2.667 & 2.664 & 3.009 & 5.904 & 4.297 & 4.303 & 6.354 & 18.289 & 8.557 & 8.597 & 20.398 \\
$2^{16}$ & 1.500 & 1.411 & 1.409 & 1.505 & 3.399 & 2.496 & 2.494 & 3.528 & 13.868 & 6.082 & 6.099 & 16.280 \\
$2^{18}$ & 0.755 & 0.727 & 0.724 & 0.766 & 1.946 & 1.418 & 1.417 & 2.010 & 10.693 & 4.273 & 4.262 & 12.744 \\
$2^{20}$ & 0.379 & 0.370 & 0.369 & 0.385 & 1.073 & 0.785 & 0.784 & 1.082 & 8.135 & 2.920 & 2.920 & 9.748 \\
\bottomrule
\end{tabular}

%% file: app_real_data.tex
\newpage
\section{Realized coverage of electricity-demand forecast intervals}
\label{app:real-data-isne}

We use real-world electricity-demand data to illustrate why temporal
dependence matters when deciding whether an observed departure from
nominal coverage is unusual.

\subsection{Setting}
\label{app:isne-setting}

This study examines realized-coverage inference for hourly electricity
demand in ISO New England. Day-ahead demand forecasts inform market
decisions and next-day operating plans.\footnote{ISO New England,
\emph{2023 Annual Markets Report}, Section~3.4.3:
\url{https://www.iso-ne.com/static-assets/documents/100011/2023-annual-markets-report.pdf}.}
For prediction intervals targeting 90\% coverage, a practitioner may
observe coverage above or below this level over the next year. The
question is whether this departure indicates miscalibration or is
consistent with sampling variation. We compare the exchangeable
reference with our block sampling method, using the same prediction
intervals throughout.

We use hourly demand and day-ahead demand forecasts from EIA's Hourly
Electric Grid Monitor, selecting the reported demand and forecast fields
for ISO New England.\footnote{Data:
\url{https://www.eia.gov/electricity/gridmonitor/}; UTC hour-ending
convention: \url{https://www.eia.gov/survey/form/eia_930/instructions.pdf}.}
The data are divided into three consecutive windows of 8,760 hours for
training, calibration, and testing. Training starts on January~1, 2023,
at 01:00 UTC. Calibration starts on January~1, 2024, at 01:00 UTC, and testing
starts on December~31, 2024, at 01:00 UTC, ending on December~31, 2025, at
00:00 UTC. Thus $n=m=8{,}760$, with adjacent calibration and test windows.

To account for calendar effects, we regress the training forecast errors
on hour-of-day by weekday/weekend indicators and two annual sine/cosine
pairs. The fitted correction is added to the published point forecast
and is fixed over calibration and testing. We use the absolute corrected
forecast error as the nonconformity score and calibrate at
$1-\alpha=0.9$. The resulting cutoff is $\widehat q_n=548.66$ MW;
each prediction interval extends this distance on either side of the
corrected forecast. The cutoff is not updated during testing.

For realized-coverage inference, we use the block sampling method in
Appendix~\ref{app:lm-procedure}, with $h=\lfloor\sqrt n\rfloor=93$ for memory estimation
and $b=\lfloor n^{2/3}\rfloor=424$. All 7,913 pairs of adjacent
calibration and test blocks, each of length $b$, lie within the
calibration window. The block distribution is scaled to the full-window
distribution using $\widehat\beta=0.632$, estimated by the procedure in
Appendix~\ref{app:lm-procedure}. We assess realized
coverage using its 2.5\% and 97.5\% quantiles and compare the result
with the exchangeable Gaussian reference, whose standard error is
$\sqrt{\alpha(1-\alpha)(1/n+1/m)}$.
Both comparisons use a two-sided 5\% test level.

\subsection{Evidence of long memory}
\label{app:isne-memory}

The calendar correction reduces calibration RMSE by 28.55\%, but
persistent dependence remains in the corrected residuals. We examine it
using detrended fluctuation analysis (DFA).
For a stationary process with autocorrelation proportional to
$k^{2H-2}$ at large lag $k$, a Hurst exponent $1/2<H<1$ corresponds
to long memory: the autocorrelations are not summable.

We apply first-order DFA to the corrected calibration residuals
\citep{kantelhardt2001detecting}.
We use 280 overlapping windows of 2,048 hours, with starts one day apart.
Within each window, we form the cumulative centered sum, remove a linear
trend in each segment, and calculate the root-mean-square fluctuation
$F(s)$ at 20 logarithmically spaced segment lengths from 16 to 512 hours,
partitioning from both ends. The slope of the log--log fit
$F(s)\propto s^H$ gives one Hurst estimate per window.

\begin{figure}[htbp]
\centering
\includegraphics[width=\linewidth]{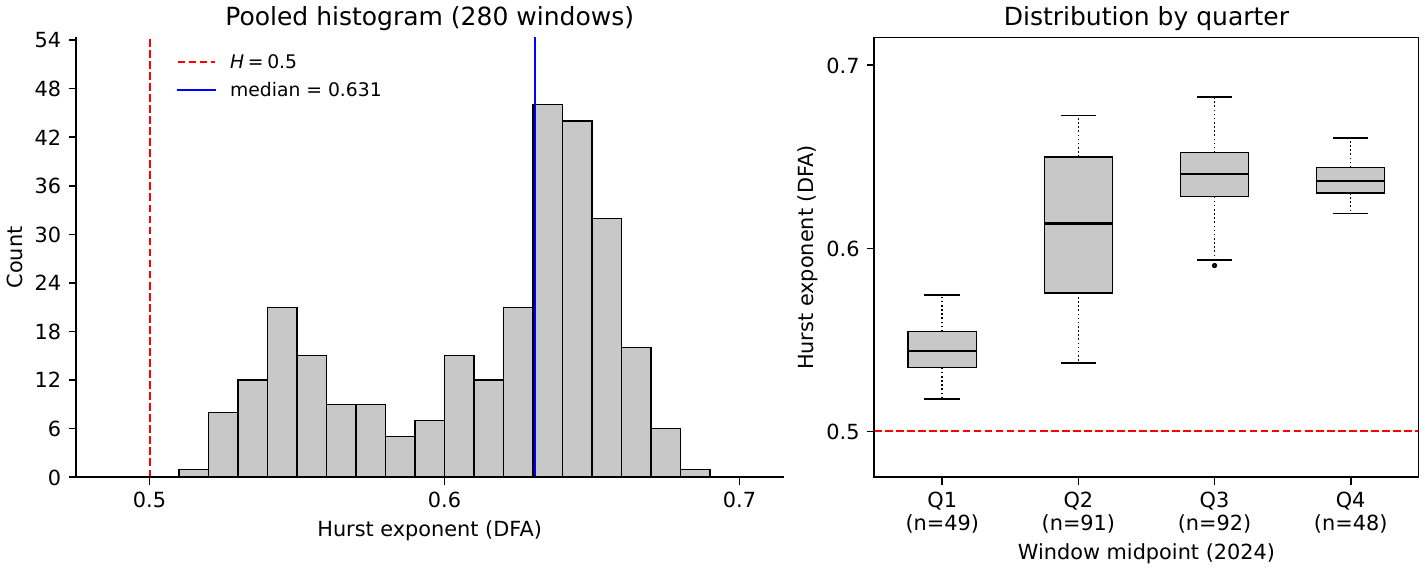}
\caption{For the electricity-demand study, Hurst estimates from corrected
calibration residuals. Left: histogram of all 280 window estimates, with
the median in blue. Right: boxplots grouped by the quarter of each
window's midpoint. Red dashed lines mark $H=0.5$. Boxes show the
interquartile range, with 1.5-IQR whiskers and outlying points.
These are summaries of overlapping windows, not confidence intervals.}
\label{fig:isne-memory}
\end{figure}

In Figure~\ref{fig:isne-memory}, all 280 window estimates exceed $0.5$,
with median 0.631 and range 0.518--0.683, providing diagnostic evidence
consistent with long memory after calendar correction.
These findings motivate accounting for persistent dependence when
assessing coverage. We use the long-memory model as a working
approximation and apply block sampling, which estimates coverage
fluctuations across the three regimes in Theorem~\ref{thm:blocks},
rather than imposing the exchangeable Gaussian reference.

\subsection{Results and interpretation}
\label{app:isne-results}

Realized coverage over the test window is 91.18\%--91.20\%, approximately
1.2 percentage points above nominal. The small range accounts for two
missing demand observations: 7,987 of the 8,758 observed responses are
covered, so the full-window fraction lies between $7987/8760$ and
$7989/8760$ without imputing or deleting any observations.

\begin{table}[htbp]
\centering
\caption{For the electricity-demand study, reference ranges for realized
coverage and two-sided testing against the 90\% nominal level. Decisions
are unchanged for every possible outcome of the two missing responses.
The reported $p$-values are approximate.}
\label{tab:isne-inference}
\begin{tabular}{lccc}
\toprule
Reference & 95\% coverage range & $p$-value & 5\% test decision\\
\midrule
Exchangeable Gaussian & [89.11\%, 90.89\%] & 0.0082--0.0095 & Reject\\
Block sampling & [84.36\%, 93.47\%] & 0.558 & Do not reject\\
\bottomrule
\end{tabular}
\end{table}

Table~\ref{tab:isne-inference} shows that the exchangeable reference
rejects at the 5\% level, whereas block sampling does not.
Figure~\ref{fig:isne-coverage} explains the difference: block sampling
assigns substantial probability to coverage values farther from 90\%,
giving a much wider reference range. The observed coverage is therefore
compatible with the fluctuations estimated by block sampling, although
it lies outside the exchangeable range.

\begin{figure}[htbp]
\centering
\includegraphics[width=0.82\linewidth]{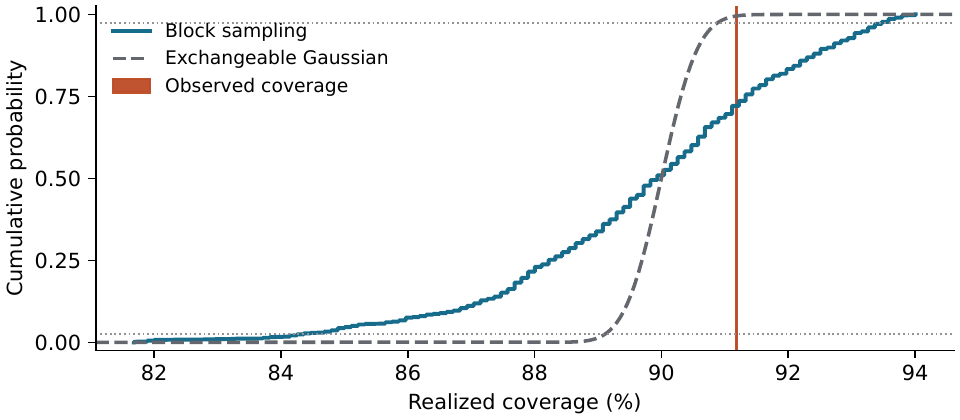}
\caption{For the electricity-demand study, the estimated cumulative
distribution of realized coverage from block sampling and the exchangeable
Gaussian reference. Dotted horizontal lines mark probabilities 0.025 and
0.975. The observed coverage (orange) lies between these quantiles of the
block distribution but above the exchangeable 97.5\% quantile.}
\label{fig:isne-coverage}
\end{figure}

This comparison illustrates the practical role of realized-coverage
inference. The forecasts, prediction intervals, and observed coverage
are identical in the two analyses; only the reference distribution
changes. Accounting for temporal dependence can therefore change
the conclusion of a coverage test.
Here, block sampling does not flag the approximately 1.2-percentage-point
overcoverage. Since our theoretical and simulation results support
better-calibrated tests when temporal dependence is taken into account,
we tend to believe that the block sampling test is more credible in this
setting and that the observed 91.2\% coverage does not provide statistically
significant evidence of a departure from the 90\% nominal level.
The result illustrates the importance of estimating
coverage uncertainty, rather than assessing a forecasting system from
its realized coverage alone.